\documentclass{amsart}
\usepackage[a4paper,twoside,inner=2cm,outer=2cm]{geometry}
\usepackage{amssymb,amsmath,amscd}

\usepackage{tikz}
\usetikzlibrary{matrix,arrows,shapes.geometric}

\usepackage{graphicx}  
\usepackage{float}  
\usepackage{multirow}  
\usepackage{booktabs}  
\usepackage{subcaption}  
\usepackage[hidelinks,colorlinks]{hyperref}
\usepackage{cleveref}

\theoremstyle{remark}
\newtheorem{dfn}{Definition}[section]
\newtheorem{exm}[dfn]{Example}

\newtheorem{rmk}[dfn]{Remark}
\theoremstyle{plain}
\newtheorem{thm}[dfn]{Theorem}
\newtheorem{prop}[dfn]{Proposition}
\newtheorem{lem}[dfn]{Lemma}
\newtheorem{cor}[dfn]{Corollary}

\newcommand*{\Cdot}{\raisebox{-1ex}{\scalebox{3}{$\cdot$}}}
\renewcommand{\leq}{\leqslant}
\renewcommand{\geq}{\geqslant}
\renewcommand{\setminus}{\smallsetminus}

\newcommand{\Z}{\mathbb{Z}}
\newcommand{\Q}{\mathbb{Q}}
\newcommand{\R}{\mathbb{R}}
\newcommand{\C}{\mathbb{C}}
\renewcommand{\P}{\mathbb{P}}

\title{Localization of certain odd-dimensional manifolds with torus actions}



\author[He]{\bfseries Chen He}

\address{
	Department of Mathematics and Physics \\ 
	North China Electric Power University \\ 
	Beijing\\
	China}

\email{che@ncepu.edu.cn}

\begin{document}


\vspace{18mm} \setcounter{page}{1} \thispagestyle{empty}

\begin{abstract}
Let a torus $T$ act smoothly on a compact smooth manifold $M$. If the rational equivariant cohomology $H^*_T(M)$ is a free $H^*_T(pt)$-module, then according to the Chang-Skjelbred Lemma, it can be determined by the $1$-skeleton consisting of the $T$-fixed points and $1$-dimensional $T$-orbits of $M$. When $M$ is an even-dimensional, orientable manifold with 2-dimensional 1-skeleton, Goresky, Kottwitz and MacPherson gave a graphic description of the equivariant cohomology. In this paper, first we revisit the even-dimensional GKM theory and introduce a notion of GKM covering, then we consider the case when $M$ is an odd-dimensional, possibly non-orientable manifold with $3$-dimensional $1$-skeleton, and give a graphic description of its equivariant cohomology.
\end{abstract}

\maketitle

\tableofcontents

\section{Introduction}
\vskip 15pt
Let a torus $T$ act smoothly on a compact smooth manifold $M$. We use $\Q$ coefficients for cohomology throughout the paper. The \textbf{$T$-equivariant cohomology} of $M$ is defined via the Borel construction $H^*_T (M) \triangleq H^*((M\times ET) / T)$, where $ET$ is the universal bundle of $T$. Fixing an identification $T\cong (S^1)^k$, we have $ET\cong (S^\infty)^k$ and $H^*_T (pt)\cong H^*(ET / T)\cong H^*((\C P^\infty)^k)=\Q[\alpha_1,\ldots,\alpha_k]$, where $\alpha_1,\ldots,\alpha_k$ of cohomological degree $2$ are the first Chern classes of the universal line bundles over $(\C P^\infty)^k$. These $\alpha_i$'s can be identified as a basis of the integral weight lattice $\mathfrak{t}^*_\Z$ of the rational dual Lie algebra $\mathfrak{t}^*_\Q$, hence the polynomial ring $\Q[\alpha_1,\ldots,\alpha_k]$ can be identified as $\mathbb{S}\mathfrak{t}^*_\Q$, the symmetric power of $\mathfrak{t}^*_\Q$. The trivial map $ M \rightarrow pt$ induces a homomorphism $ H^*_T(pt)\rightarrow H^*_T(M)$ and gives $H^*_T(M)$ an $H^*_T (pt)$-algebra structure. 

For $p\in M$, denote its $T$-orbit by $\mathcal{O}_p$. Set the \textbf{$i$-th skeleton} $M_i=\{p\in M \mid \textup{dim}\,\mathcal{O}_p\leq i\}$, then we have a $T$-equivariant filtration of closed subsets $M_0 \subseteq M_1 \subseteq \cdots \subseteq M_{\textup{dim}\,T}=M$, where the $0$-skeleton $M_0$ is the fixed-point set $M^T$. 

If $H^*_T (M)$ is a free $H^*_T (pt)$-module, Chang and Skjelbred \cite{CS74} proved that the equivariant cohomology $H^*_T(M)$ can be described as a sub-ring of $H^*_T(M^T)$, subject to certain relations determined by the $1$-skeleton $M_1$. Goresky, Kottwitz and MacPherson \cite{GKM98} considered certain torus actions on complex projective manifolds such that the fixed-point set $M^T$ is finite and the $1$-skeleton $M_1$ is a finite union of $S^2$'s. They proved that the cohomology $H^*_T(M)$ can be described in terms of congruence relations on a graph determined by the $1$-skeleton $M_1$. Since then, various GKM-type theorems were proved, for instance, by Brion \cite{Br97} on equivariant Chow groups, by Guillemin\&Zara \cite{GZ01} on abstract GKM graph theory, by Knutson\&Rosu \cite{KR03}, Vezzosi\&Vistoli \cite{VV03} on equivariant K-theory, and by Guillemin\&Holm \cite{GH04} on certain Hamiltonian torus actions on symplectic manifolds with non-isolated fixed points. Recent generalizations of GKM-type theorem were given by Goertsches, Nozawa\&T\"{o}ben \cite{GNT12} on certain Cohen-Macaulay actions on K-contact manifolds, and by Goertsches\&Mare \cite{GM14a} on actions of non-abelian groups.  

In this paper, first we revisit the even-dimensional GKM theory and introduce a notion of GKM covering, then we consider the case when $M$ is an odd-dimensional, possibly non-orientable manifold with $3$-dimensional $1$-skeleton, and give a graphic description of its equivariant cohomology.

\textbf{Acknowledgement} An early draft of this paper was part of the author's PhD thesis at Northeastern University, Boston. The author thanks Shlomo Sternberg for teaching him symplectic geometry. The author thanks Victor Guillemin and Jonathan Weitsman for guidance. The author thanks Catalin Zara for many useful discussions and carefully reading the manuscript, and thanks Volker Puppe, Oliver Goertsches and Liviu Mare for helpful suggestions. The author also thanks the anonymous referees for their excellent comments. The author thanks the Ling-Ma graduate research fund at Northeastern University and the China Postdoctoral Science Foundation (Grant No.\,2018T110083) for the generous support.

\vskip 20pt
\section{Torus actions and equivariant cohomology}
\vskip 15pt
We will recall some definitions and classical theorems regarding torus actions and equivariant cohomology. For general reference, see \cite{B72,H75,K91,AP93}.

\subsection{Torus actions}\label{subsec:T-action}
Throughout the paper, unless otherwise mentioned, a manifold $M$ is assumed to be smooth, compact, connected and boundaryless, but possibly non-orientable. Let a torus $T$ act smoothly, effectively on a manifold $M$. If $M$ is orientable, then we fix an orientation and assume that the $T$-action preserves the orientation. 

\subsubsection{Fixed-point set and isotropy weights}\label{subsub:Fixed-point}
For a point $p$ in a connected $T$-fixed component $C\subseteq M^T$, there is the \textbf{isotropy representation} of $T$ on the tangent space $T_p M$, which splits into weight spaces $T_p M = V_0 \oplus V_{[\lambda_1]} \oplus \cdots \oplus V_{[\lambda_r]}$ that holds for any $p\in C$. Each nonzero weight $[\lambda_i]\in \mathfrak{t}^*_\Z/\{\pm 1\}$ is determined only up to sign, and we have $V_{[\lambda_i]}\cong W_i\otimes_\R \R^2_{[\lambda_i]}$, where $W_i$ is a real vector space and $\R^2_{[\lambda_i]}$ is the irreducible real $T$-representation of weight $[\lambda_i]$.

Since $M$ is assumed to be compact, there is a $T$-invariant Riemannian metric. The exponential map at $p$, restricted to the subspace $V_0 \subset T_p M$, gives a local submanifold structure of a fixed component $C$ at $p$. Comparing the isotropy weight splitting with the tangent-normal splitting $T_p M = T_p C \oplus N_p C$ along $C$, we get $T_p C = V_0$ and $N_p C = V_{[\lambda_1]} \oplus \cdots \oplus V_{[\lambda_r]}$. The dimensions of $M$ and $C$ are of the same parity. If $M^T$ is non-empty, then for even $\mathrm{dim}\, M$, the smallest possible components of $M^T$ could be isolated points; for odd $\mathrm{dim}\, M$, the smallest possible components of $M^T$ could be isolated circles.

\subsubsection{Orientations}\label{subsub:orientation}
Fixing a sign for $[\lambda_i]$ as $\lambda_i \in \mathfrak{t}^*_\Z$ is equivalent to identifying $\R^2_{[\lambda_i]}$ as the irreducible complex $T$-representation $\C_{\lambda_i}$. This gives a complex structure and hence an orientation for $V_{[\lambda_i]}$. From now on, we assume that, at every $T$-fixed component, a sign of $[\lambda_i]$ and hence an orientation of $V_{[\lambda_i]}$ have been chosen and we will write them as $\lambda_i$ and $V_{\lambda_i}$. If $M$ has a $T$-invariant almost complex structure, then the signs of the isotropy weights are canonically determined. 

The orientations on $V_{\lambda_i}$'s give an orientation on $N_p C$. If $M$ is oriented, then the tangent-normal splitting $T_p M = T_p C \oplus N_p C$ induces an orientation on $T_p C$. If we have prechosen an orientation for $C$, then we can compare this prechosen orientation with the induced orientation to get a sign. On the other hand, even if $M$ is non-orientable, it is still possible to have some orientable components $C\subseteq M^T$. 

\subsubsection{Sub-actions and residual actions}
For any subtorus $K$ of $T$, we get two more actions automatically: the \textbf{sub-action} of $K$ on $M$ and the \textbf{residual action} of $T/K$ on $M^K$ using the fact that $T$ is abelian. Moreover, we have $(M^K)^{T/K}=M^T$. 

\subsection{Some basics of equivariant cohomology}

Given an action $T\curvearrowright M$, one can compare $H^*_T(M)$ with $H^*_T(M^T)$:

\begin{thm}[Borel Localization Theorem]\label{BoL}
	The restriction map 
	$
	H^*_T(M)\longrightarrow H^*_T(M^T)
	$
	is an $H^*_T (pt)$-module isomorphism modulo $H^*_T (pt)$-torsions.
\end{thm}

\begin{dfn}
	An action $T\curvearrowright M$ is \textbf{equivariantly formal} if $H^*_T (M)$ is a free $H^*_T (pt)$-module. 
\end{dfn}

For an equivariantly formal $T$-action, the restriction map embeds $H^*_T(M)$ into $H^*_T(M^T)$. 

\begin{cor}[Existence of fixed points]\label{Nonempty}
If an action $T\curvearrowright M$ is equivariantly formal, then the fixed-point set $M^T$ is non-empty.
\end{cor}

Moreover, the embedded image of $H^*_T(M)$ can be described in the following way:

\begin{thm}[Chang-Skjelbred Lemma, \cite{CS74}]\label{Chang}
	If an action $T\curvearrowright M$ is equivariantly formal, then the equivariant cohomology $H^*_T (M)$ only depends on the fixed-point set $M^T$ and the $1$-skeleton $M_1$:
	\[
	H^*_T(M) \cong  H^*_T(M_1) \cong \bigcap_{\mathrm{codim}\,K=1} \Big( \mathrm{Im}\big( H^*_T(M^K) \rightarrow  H^*_T(M^T)\big)\Big)
	\]
	where the intersection is taken over all (finitely many) codim-$1$ subtori $K$ that are also the identity components of some stabilizers of the action $T\curvearrowright M$.
\end{thm}

\begin{rmk}
	The above version of Chang-Skjelbred Lemma was due to Goresky, Kottwitz and MacPherson \cite{GKM98}, also see Tolman and Weitsman \cite{TW99}, Goldin and Holm \cite{GoHo01}. 
\end{rmk}

\begin{rmk}
	A more general result, the Atiyah-Bredon long exact sequence, appeared earlier in Atiyah's 1971 lecture notes \cite{A74} for equivariant K-theory and later in Bredon's work \cite{B74} for equivariant cohomology. Franz and Puppe \cite{FP11} generalized the Chang-Skjelbred Lemma to some other coefficient rings. 
\end{rmk}

Equivariant formality is equivalent to the degeneracy at the $E_2$ page of the Leray-Serre spectral sequence of the fibration $M \hookrightarrow (M \times ET) /T \rightarrow BT$. There is a useful criterion for equivariant formality:

\begin{thm}[Total Betti numbers and equivariant formality, see {\cite[p.\,210, Thm\,3.10.4]{AP93}}]\label{CohomIneq}
If $T$ acts on $M$, then $\sum_i \mathrm{dim}\,H^i(M^T) \leq \sum_i \mathrm{dim}\,H^i(M)$, where the equality holds if and only if the action is equivariantly formal.
\end{thm}

Equivariant and ordinary cohomology can be calculated using Morse theory, for example see \cite{AuBr95}. A sufficient condition for equivariant formality in the presence of a Morse-Bott function is that:
\begin{prop}
If a $T$-manifold $M$ has a $T$-invariant Morse-Bott function $f$ such that the critical submanifold $Crit(f)$ is the fixed-point set $M^T$, then the action $T\curvearrowright M$ is equivariantly formal and the function $f$ is perfect (i.e. $\sum_i \mathrm{dim}\,H^i(M)= \sum_i \mathrm{dim}\,H^i(Crit(f))$).
\end{prop}
\begin{proof}
Theorem\,\ref{CohomIneq} gives $\sum_i \mathrm{dim}\,H^i(M^T) \leq \sum_i \mathrm{dim}\,H^i(M)$. The cohomology $H^*(M)$ can be computed from the Morse-Bott-Witten cochain complex generated on the critical submanifold $Crit(f)$, hence $\sum_i \mathrm{dim}\,H^i(M)\leq \sum_i \mathrm{dim}\,H^i(Crit(f))$.  By our assumption regarding the critical point set and the fixed-point set, we have the equality $\sum_i \mathrm{dim}\,H^i(M^T)=\sum_i \mathrm{dim}\,H^i(Crit(f))$ which forces the previous two inequalities to be equalities and hence implies the equivariant formality of $T\curvearrowright M$ and the perfection of $f$.
\end{proof}

\begin{exm}\label{exm:FormalityOfSymp}
When $M$ is equipped with a symplectic form, a Hamiltonian $T$-action and a moment map $\mu:M\rightarrow \mathfrak{t}^*$, then $\mu^\xi$ gives a Morse-Bott function for any $\xi \in \mathfrak{t}$. Let $\xi$ be generic such that the one-parameter subgroup generated by $\xi$ is dense in $T$, then we have $Crit(\mu^\xi)=M^T$. Hence the above Proposition implies Kirwan's theorem that $M$ is $T$-equivariantly formal and $\mu^\xi$ is perfect. Bozzoni and Goertsches \cite{BG19} observed that this argument also works for certain class of Hamiltonian torus actions on cosymplectic manifolds.
\end{exm}

Restricting to any subtorus $K$ of $T$ acting on $M$, we get

\begin{prop}[Inheritance of equivariant formality] \label{InheritFormality}
An action $T\curvearrowright M$ is equivariantly formal if and only if for any subtorus $K$ of $T$, both the sub-action $K\curvearrowright M$ and the residual action $T/K\curvearrowright M^K$ are equivariantly formal.
\end{prop}
\begin{proof}
The sub-action $K\curvearrowright M$ gives the inequality $
\sum_i \mathrm{dim}\,H^i(M^K) \leq \sum_i \mathrm{dim}\,H^i(M)$. Since $(M^K)^{T/K}=M^T$, the residual action $T/K\curvearrowright M^K$ gives the inequality 
$
{\sum_i \mathrm{dim}\,H^i(M^T)\leq \sum_i \mathrm{dim}\,H^i(M^K)}.
$
Thus, the equality $\sum_i\mathrm{dim}\, H^i(M^T)=\sum_i \mathrm{dim}\,H^i(M)$ holds if and only if both intermediate equalities $\sum_i \mathrm{dim}\,H^i(M^K) = \sum_i \mathrm{dim}\,H^i(M)$ and $\sum_i \mathrm{dim}\,H^i(M^T)= \sum_i \mathrm{dim}\,H^i(M^K)$ hold, which is just a restatement of the proposition. 
\end{proof}

\begin{cor}[Inheritance of fixed points]\label{InheritFix}
If an action $T\curvearrowright M$ is equivariantly formal, then for any subtorus $K$ of $T$, every connected component of $M^K$ has $T$-fixed points.
\end{cor}
\begin{proof}
By the inheritance of equivariant formality, the residual action of $T/K$ on any connected component $C$ of $M^K$ is also equivariantly formal. Then by the existence of fixed points, $C^{T}=C^{T/K}$ is non-empty.
\end{proof}

\vskip 20pt
\section{GKM theory in even dimensions}
\vskip 15pt
Goresky, Kottwitz and MacPherson \cite{GKM98} originally considered their theory for certain complex projective manifolds with torus actions. Goertsches and Mare \cite{GM14a} observed that those ideas also work for certain possibly non-orientable, even-dimensional manifolds with torus action.

\subsection{GKM condition in even dimensions}
Goresky, Kottwitz and MacPherson \cite{GKM98} considered the smallest possible fixed-point set and $1$-skeleton.

\begin{dfn}[GKM condition in even dimensions]\label{EvenGKMCond}
An action $T\curvearrowright M^{2n}$ is \textbf{GKM} if it is equivariantly formal and the following is satisfied
\begin{itemize}
\item[(1)] The fixed-point set $M^T$ is a non-empty set of isolated points. 
\item[(2)] The $1$-skeleton $M_1$ is $2$-dimensional. Or equivalently, at each fixed point $p\in M^T$, the non-zero weights $\lambda_1,\ldots,\lambda_n \in \mathfrak{t}^*_\Z$ of the isotropy $T$-representation $T\curvearrowright T_pM$ are pair-wise linearly independent.
\end{itemize}
\end{dfn}

By Condition (1), we get $H^*_T(M^T)=\bigoplus_{p\in M^T}\mathbb{S}\mathfrak{t}^*_\Q$.

By Condition (2), for any isotropy weight $\lambda$ at a fixed point $p$, if we set $T_{\lambda}$ as the codimension-1 subtorus of $T$ with the Lie subalgebra $\mathrm{Ker}\,\lambda \subset \mathfrak{t}$, then the component $C_{\lambda}$ of $M^{T_{\lambda}}$ containing $p$ is $2$-dimensional and has an effective residual action of $T/T_{\lambda}$. Choose a $\xi\in \mathfrak{t}$ such that $\lambda(\xi)> 0$, then $\xi$ generates a non-vanishing tangent vector field on $T/T_{\lambda}$. This gives an orientation on $T/T_{\lambda}$ and identifies it as an $S^1$.

\subsection{Geometry and equivariant cohomology of 2d $S^1$-manifolds}
There is a well known classification of $S^1$-actions on compact surfaces with non-empty fixed-point sets. The following result, that the author learned from Audin's book \cite{Au04} but could be traced back much earlier, can be proved by the differentiable slice theorem and other methods.

\begin{lem}[cf. \cite{Au04} p.\,20]\label{2dFormal_1}
	If $S^1$ acts effectively on a closed surface $M$ with a non-empty fixed-point set, then $M$ is $S^1$-equivariantly diffeomorphic to one of the following two:
	\begin{itemize}
		\item $S^2$ with two fixed points;
		\item $\R P^2$ with one fixed point, and an exceptional orbit $S^1/{\Z_2}$
	\end{itemize}
	where $\R P^2$ as the $\Z_2$-quotient of $S^2$, has an induced $S^1$-action from $S^2$.  
\end{lem}

Write $H^*_{S^1}(pt)\cong H^*(BS^1)\cong H^*(\C \P^\infty) \cong \Q[\alpha]$ as a polynomial ring in the degree-$2$ variable $\alpha$. Using the equivariant Mayer-Vietoris sequence of $S^2$ viewed as the union of the north and the south hemispheres, we get
\begin{align*}
H^*_{S^1}(S^2) &\cong \big\{ (f_N,f_S)\in \mathbb{Q}[\alpha]\oplus\mathbb{Q}[\alpha]\mid f_N(0)=f_S(0)\big\}.
\end{align*}
The antipodal covering map $\Z_2\rightarrow S^2 \rightarrow \R P^2$ induces a $\Z_2$-action on $H^*_{S^1}(S^2)$ that swaps the tuple $(f_N,f_S)$ to $(f_S,f_N)$. Then we have
\[
H^*_{S^1}(\R P^2) \cong H^*_{S^1}(S^2)^{\Z_2} \cong  \big\{ (f,f)\in \mathbb{Q}[\alpha]\oplus\mathbb{Q}[\alpha]\big\}\cong \mathbb{Q}[\alpha].
\]
The $S^1$-actions on $S^2$ and $\R P^2$ are both equivariantly formal. 

Back to the component $C_\lambda$ of $M^{T_\lambda}$ with the residual action of the circle $T/T_\lambda$, it can only be a sphere or a projective plane, denoted by $S^2_\lambda$ or $\R P^2_\lambda$, such that the codim-$1$ subtorus $T_\lambda$ acts on it trivially and the residual circle $T/T_\lambda$ acts on it equivariantly formally. We have
\begin{align}\label{eq:S2RP2}\tag{$\ast$}
\begin{split}
H^*_T(S^2_{\lambda}) &\cong H^*_{T/T_\lambda}(S^2_{\lambda})\otimes H^*_{T_\lambda}(pt)\cong\big\{ (f_N,f_S)\in \mathbb{S}\mathfrak{t}^*_\Q\oplus\mathbb{S}\mathfrak{t}^*_\Q\mid f_N \equiv f_S \mod{\lambda}\big\}\\
H^*_T(\R P^2_{\lambda}) &\cong H^*_{T/T_\lambda}(\R P^2_{\lambda})\otimes H^*_{T_\lambda}(pt)\cong\mathbb{S}\mathfrak{t}^*_\Q
\end{split}
\end{align}
where the congruence relation $f_N \equiv f_S \mod{\lambda}$ means that $f_N - f_S$ is divisible by $\lambda$.

\subsection{GKM graphs and the GKM theorem in even dimensions}
In the $1$-skeleton $M_1$, each $S^2$ has two fixed points, and each $\R P^2$ has one fixed point. This observation leads to a graphic representation of the relations between $M^T$ and $M_1$, given by Goresky, Kottwitz and MacPherson in the orientable case and by Goertsches and Mare in the possibly non-orientable case.

\begin{dfn}\label{dfn:evenGKMGraph}
	The \textbf{GKM graph} of a GKM action $T\curvearrowright M^{2n}$ consists of: 
	\begin{description}
		\item[Vertices] There are two types of vertices:
		\begin{description}
			\item[$\bullet$] for each fixed point in $M^T$,
			\item[$\star$] for each $\R P^2 \in M_1$.
		\end{description}
		\item[Edges$\,\&\,$Weights] 
		For each $S^2_{\lambda}$, a solid edge with weight ${\lambda}$ joins the two $\bullet$'s that represent its two fixed points. For each $\R P^2_{\mu}$, a dotted edge with weight ${\mu}$ joins its $\star$ to the $\bullet$ that represents its fixed point.
	\end{description}
\end{dfn}

The following notion was introduced by Guillemin and Zara \cite{GZ01}.

\begin{dfn}
	Let $T\curvearrowright M^{2n}$ be a GKM action with the GKM graph $\mathcal{G}$ consisting of the vertex set $V=M^T$ and the weighted edge set $E$. The \textbf{cohomology ring of the GKM graph $\mathcal{G}$}, denoted by $H^*(\mathcal{G})$, is the set 
	$$\big\{ f: V\rightarrow \mathbb{S}\mathfrak{t}^*_\Q \mid f_p \equiv f_q \mod{\lambda} \quad \mbox{for each solid edge $\overline{pq}$ with weight $\lambda$ 
		in $E$}\big\}$$
	which has a canonical $\mathbb{S}\mathfrak{t}^*_\Q$-algebra structure.
\end{dfn}

Combining the Chang-Skjelbred Lemma and the equivariant cohomology of $S^2$ and $\R P^2$ as in \eqref{eq:S2RP2}, we see:

\begin{thm}[GKM theorem in even dimensions, {\cite[p.\,26 Thm\,1.2.2]{GKM98}, \cite[p.\,7 Thm\,3.6]{GM14a}}]\label{thm:EvenGKM}
	If an effective $T$-action on an even-dimensional, possibly non-orientable manifold $M^{2n}$ is GKM with the graph $\mathcal{G}$, then there is an $\mathbb{S}\mathfrak{t}^*_\Q$-algebra isomorphism:
	\[
	H^*_T(M) \cong H^*(\mathcal{G}).
	\]
\end{thm}

\begin{rmk}
	The $\R P^2$'s in the $1$-skeleton $M_1$ do not contribute to the congruence relations. We can erase all the dotted edges in the GKM graph, and call the remaining graph an \textbf{effective} GKM graph. However, if we want to get a GKM-type theorem for much subtler coefficient rings like $\Z$, the $\R P^2$'s in the $1$-skeleton $M_1$ and their corresponding dotted edges in the GKM graph are as crucial as the $S^2$'s and their corresponding solid edges.
\end{rmk}

\begin{rmk}
	The notion of GKM bundles introduced by Guillemin, Sabatini and Zara  \cite{GSZ12} can be naturally generalized to the possibly non-orientable case.
\end{rmk}

\begin{exm}\label{exm:EvenSphere}
For the sphere $S^{2n}=\{(x,z_1,\ldots,z_n)\in \R \oplus \C^n \mid x^2+||z_1||^2+\cdots+||z_n||^2=1\}$, let $T^n$ act on it by $(e^{i\theta_1},\ldots,e^{i\theta_n})\cdot (x,z_1,\ldots,z_n) = (x,e^{i\theta_1}z_1,\ldots,e^{i\theta_n}z_n)$ with the fixed-point set $(S^{2n})^{T^n}=\{(\pm 1,0,\ldots,0)\}$. Since $\mathrm{dim}\,H^*((S^{2n})^{T^n})=2=\mathrm{dim}\,H^*(S^{2n})$, the $T^n$-action on $S^{2n}$ is equivariantly formal by the Formality Criterion Theorem \ref{CohomIneq}. Let $\alpha_1,\ldots,\alpha_n$ be the standard integral basis of $\mathfrak{t}^*_\Z=\Z^n$. Both fixed points have the isotropy weights $\alpha_1,\ldots,\alpha_n$. This means the action is GKM and the GKM graph consists of two $\bullet$-vertices joined via $n$ edges of weights $\alpha_1,\ldots,\alpha_n$. The equivariant cohomology ring is
\begin{align*}
H_{T^n}^*(S^{2n})
&=\{(f,g)\in \mathbb{S}\mathfrak{t}^*_\Q\oplus \mathbb{S}\mathfrak{t}^*_\Q \mid f \equiv g \mod \prod_{i=1}^{n}\alpha_i \}.
\end{align*}
For every such pair $(f,g)$ satisfying $f \equiv g \mod \prod_{i=1}^{n}\alpha_i$, we can write $f-g=h\cdot \prod_{i=1}^{n}\alpha_i$ for an $h\in \mathbb{S}\mathfrak{t}^*_\Q$, hence $f=(f+g)/{2}+(h/2)\cdot\prod_{i=1}^{n}\alpha_i$ and $g=(f+g)/{2}-(h/2)\cdot\prod_{i=1}^{n}\alpha_i$. Let $e=(\prod_{i=1}^{n}\alpha_i,-\prod_{i=1}^{n}\alpha_i)\in\mathbb{S}\mathfrak{t}^*_\Q\oplus \mathbb{S}\mathfrak{t}^*_\Q$, then $\{(1,1),e\}$ is an $\mathbb{S}\mathfrak{t}^*_\Q$-module basis for $H_{T^n}^*(S^{2n})$, and satisfies the relation $e^2=(\prod_{i=1}^{n}\alpha^2_i)\cdot(1,1)$. Write $\mathbb{S}\mathfrak{t}^*_\Q=\Q[\alpha_1,\ldots,\alpha_n]$, we have a ring isomorphism 
\[
H_{T^n}^*(S^{2n})\cong \frac{\Q[\alpha_1,\ldots,\alpha_n;e]}{\langle e^2 = \prod_{i=1}^{n}\alpha^2_i\rangle}.
\]
If we replace $e$ by another generator $e'=(\prod_{i=1}^{n}\alpha_i,0)$ which satisfies ${e'}^2=(\prod_{i=1}^{n}\alpha^2_i,0)=(\prod_{i=1}^{n}\alpha_i)\cdot e'$, we then have another ring isomorphism
\[
H_{T^n}^*(S^{2n})\cong \frac{\Q[\alpha_1,\ldots,\alpha_n;e']}{\langle {e'}^2 = (\prod_{i=1}^{n}\alpha_i) \cdot e'\rangle}.
\]

\end{exm}

\begin{rmk}
	The 2n-dimensional sphere $S^{2n}$ under the standard $T^n$-action can be viewed as a torus manifold as defined by Hattori and Masuda \cite{HM03}. The general results of integral equivariant cohomology rings of torus manifolds were given by Masuda and Panov \cite{MP06}, and its GKM theory were given by Maeda, Masuda and Panov \cite{MMP07}. 
\end{rmk}

\begin{exm}
	$\R P^{2n}$ is the quotient of $S^{2n}$ by the $\Z_2$-action $e^{\pi i}\cdot (x,z_1,\ldots,z_n)=(-x,-z_1,\ldots,-z_n)$. It inherits a $T^n$-action from $S^{2n}$. The fixed-point set is $(\R P^{2n})^{T^n}=\{(\pm 1,0,\ldots,0)\}/\Z_2$, a single point. Since $\mathrm{dim}\,H^*((\R P^{2n})^{T^n})=1=\mathrm{dim}\,H^*(\R P^{2n})$, the $T^n$ action on $\R P^{2n}$ is equivariantly formal by the Formality Criterion Theorem \ref{CohomIneq} with the isotropy weights $\alpha_1,\ldots,\alpha_n$ at the only fixed point. This means the action is GKM and the GKM graph consists of a $\bullet$ joined via $n$ dotted edges of weights $\alpha_1,\ldots,\alpha_n$ to n $\star$'s. The effective GKM graph is a single $\bullet$ without edges. The equivariant cohomology is then $H_{T^n}^*(\R P^{2n})\cong H_{T^n}^*(pt)= \mathbb{S}\mathfrak{t}^*_\Q$. This example was due to \cite{GM14a} p.7 Example 3.7.
\end{exm}

\subsection{GKM covering}\label{subsec:GKMcover}
	Let $\tilde{M}^{2n} \rightarrow M^{2n}$ be a $T$-equivariant covering with a finite covering group $\Gamma$. See \cite[Sec. I.9]{B72} for a general discussion on equivariant covering. If the $T$-action on $\tilde{M}$ is GKM, then according to the even GKM Theorem \ref{thm:EvenGKM}, the equivariant cohomology $H^*_T(\tilde{M})$ concentrates in even degrees, so does its ordinary cohomology $H^*(\tilde{M})$. Since $H^*(M) \cong H^*(\tilde{M})^\Gamma$, the ordinary cohomology $H^*(M)$ also concentrates in even degrees, which implies the collapse of the Leray-Serre spectral sequence of the fibration $M \hookrightarrow (M \times ET) /T \rightarrow BT$ at the $E_2$ page $H^*(BT)\otimes H^*(M)$. Therefore, the $T$-action on $M$ is equivariantly formal. The isotropy weights at the $T$-fixed points of $M$ are inherited from $\tilde{M}$, hence the $T$-action on $M$ is also GKM. Restricting the covering to the fixed points and $1$-skeleta we get the coverings:
	\begin{align*}
	\Gamma\longrightarrow\tilde{M}^T \longrightarrow M^T,\qquad \qquad \Gamma\longrightarrow\tilde{M}_1 \longrightarrow M_1.
	\end{align*}
	
	\begin{dfn}[Finite coverings/quotients of GKM graphs] \label{dfn:GKMcover}
		Given a $T$-equivariant finite covering $\Gamma \rightarrow \tilde{M}^{2n} \rightarrow M^{2n}$ of GKM manifolds, denote their GKM graphs by $\tilde{\mathcal{G}},\,\mathcal{G}$, then there is a $\Gamma$-action on $\tilde{\mathcal{G}}$ and we can view $\mathcal{G}$ as a quotient graph $\tilde{\mathcal{G}}/\Gamma$ in the following sense: 
		\begin{enumerate}
			\item The $\Gamma$-orbits of $\bullet$ vertices in $\tilde{\mathcal{G}}$ one-to-one correspond to the $\bullet$ vertices in $\mathcal{G}$.
			\item The free $\Gamma$-orbits of solid edges in $\tilde{\mathcal{G}}$ one-to-one correspond to solid edges in $\mathcal{G}$.
			\item The $\Gamma$-orbits of $\star$ vertices and dotted edges in $\tilde{\mathcal{G}}$ form a part of the $\star$ vertices and dotted edges in $\mathcal{G}$.
			\item The non-free $\Gamma$-orbits of solid edges in $\tilde{\mathcal{G}}$ form the remaining $\star$ vertices and dotted edges in $\mathcal{G}$.
		\end{enumerate} 
	\end{dfn}
	
	\begin{rmk}
		The (1),\,(2),\,(3) are natural. To understand (4), suppose a solid edge in $\tilde{\mathcal{G}}$ has a nontrivial $\Gamma$-stabilizer, then the represented $S^2$ will be folded by that stabilizer to form a $\R P^2$ which produces a $\star$ and a dotted edge in $\mathcal{G}$.
	\end{rmk}
	
	\begin{rmk}
		We have the equivalences $H^*_T(M)\cong H^*(\mathcal{G}) \cong H^*_T(\tilde{M})^\Gamma \cong H^*(\tilde{\mathcal{G}})^\Gamma$.
	\end{rmk}

	\begin{rmk}
		The above definition also makes sense for finite coverings/quotients of abstract GKM graphs that do not necessarily come from actual GKM manifolds.
	\end{rmk}
	
	\begin{rmk}
		Another idea of obtaining symmetry on GKM graph has been given by Kaji \cite{Ka15} on GKM $T$-manifolds with extended Lie group actions.
	\end{rmk}
As an application of the notion of GKM covering, we can extend Guillemin, Holm and Zara's GKM descriptions \cite[p.\,28 Thm 2.4]{GHZ06} of a certain class of homogeneous spaces to a slightly larger class. 

Let $G$ be a compact, possibly non-connected Lie group with a maximal torus $T$, and $K$ be a closed subgroup of $G$ containing $T$. Denote the sets of positive roots of $G$ and $K$ by $\bigtriangleup^+_{G}$ and $\bigtriangleup^+_{K}$. We use the definition of the Weyl group $W_G\triangleq N_G(T)/T$, where $N_G(T)$ is the normalizer of $T$ in $G$. Let $G_0$ be the identity component of $G$. Since the inclusion $G_0\subseteq G$ is normal, the quotient $G/G_0$ is a finite group. The exact sequence of groups $G_0\hookrightarrow G \twoheadrightarrow G/G_0$ descends to the exact sequence of finite groups $W_{G_0}\hookrightarrow W_{G} \twoheadrightarrow G/G_0$ using the conjugacy properties of $T$ in $G$. We see that the Weyl group $W_G$ is generated not only by the Weyl reflections $\sigma_\lambda,\, \lambda \in \bigtriangleup^+_{G}$, but also by $G/G_0$. 

\begin{prop}\label{prop:GHZ}
	The natural left action $T \curvearrowright G/K$ is GKM with a GKM graph $\mathcal{G}_{\mbox{\tiny $G/K$}}$ such that
	\begin{enumerate}
		\item The vertex set is $W_G/W_K$.
		\item The labelled $S^2$-edges at any $[w]\in W_G/W_K$ are
		\begin{figure}[H]
			\centering
			\begin{tikzpicture}[scale=1.5]
			\path (0:0cm) node[draw,fill,circle,inner sep=0pt,minimum size=4pt,label=left:{$[w]$}] (v0) {};
			\path (0:1cm) node[draw,fill,circle,inner sep=0pt,minimum size=4pt,label=right:{$[w\sigma_\lambda]$}] (v1) {};
			
			\draw (v0) -- node[auto] {$w\lambda$} (v1);
			\end{tikzpicture}
		\end{figure}
		\noindent for all $\lambda \in \bigtriangleup^+_{G} \setminus \bigtriangleup^+_{K}$ with $\sigma_\lambda \not \in W_K$.
		\item The labelled $\R P^2$-edges at any $[w]\in W_G/W_K$ are
		\begin{figure}[H]
			\centering
			\begin{tikzpicture}[scale=1.5]
			\path (0:0cm) node[draw,fill,circle,inner sep=0pt,minimum size=4pt,label=left:{$[w]$}] (v0) {};
			\path (0:1cm) node[draw, fill, star, star points=5, star point ratio=2.25, inner sep=1pt] (v1) {};
			
			\draw (v0)[densely dashed] -- node[auto] {$w\mu$} (v1);
			\end{tikzpicture}
		\end{figure}
		\noindent for all $\mu \in \bigtriangleup^+_{G} \setminus \bigtriangleup^+_{K}$ with $\sigma_\textit{\tiny $\mu$} \in W_K$.
	\end{enumerate}
\end{prop}
\begin{proof}
(Sketch) Guillemin, Holm and Zara \cite{GHZ06} assumed that $G$ is semisimple, connected, and $K$ is connected, but also suggested dropping that assumption using covering space arguments. Following the proof in \cite{GHZ06}, it can be similarly verified that $(G/K)^T=W_G/W_K$, and that the tangent space at the identity coset $eK \in G/K$ is a $T$-representation 
$
T_{eK} G/K = \bigoplus_{\lambda \in \bigtriangleup^+_{G} \setminus \bigtriangleup^+_{K}} \C_\lambda
$.
Hence we get the vertex set and the labelled edge set of the GKM graph of $G/K$. As for the determination of an edge as either an $S^2$ or $\R \P^2$, it depends on whether the edge formed from a reflection $\sigma_\lambda$ is folded when taking the quotient of $W_G$ by $W_K$, i.e. whether $\sigma_\lambda$ is actually an element of $W_K$.
\end{proof}

\begin{exm}
	Denote the real and oriented Grassmannians of $k$-dimensional subspaces in $n$-dimensional real spaces by $G_k(\R^n)$ and $\tilde{G}_k(\R^n)$. They are of the dimension $k(n-k)$. There is a natural $\Z_2$-cover $\tilde{G}_k(\R^n) \rightarrow G_k(\R^n)$ by forgetting the orientations of the oriented $k$-dimensional subspaces. When these Grassmannians are even-dimensional, they are equipped with canonical GKM torus actions which are compatible with the $\Z_2$-covering map. For instance, let $T^2$ act on $\R^5$ such that each $S^1$-factor of $T^2$ respectively rotates each $\R^2$-summand of $\R^5=\R_1^2\oplus\R_2^2\oplus \R$. This real $T^2$-representation induces a canonical $T^2$-action on the $6$-dimensional Grassmannian $G_2(\R^5)$ such that for $t\in T^2$ and $V\in G_2(\R^5)$, we define $t\cdot V$ as the image $t(V)$ of the isomorphism $t:\R^5\rightarrow \R^5$. The fixed-points are the real $2$-dimensional $T^2$-subrepresentations of the representation $T^2\curvearrowright \R^5$, i.e. $G_2(\R^5)^{T^2}=\{\R_1^2,\R_2^2\}$. We denote these two $\R^2$-summands by $V_1,V_2$. Similarly, we get a canonical $T^2$-action on $\tilde{G}_2(\R^5)$ with the fixed-point set $\tilde{G}_2(\R^5)^{T^2}=\{V_{1_+},V_{1_-},V_{2_+},V_{2_-}\}$, where the underlying space of $V_{i_\pm}$ is $V_i$. To understand the $1$-skeleta, one can go further to consider the $2$-dimensional real $S^1$-subrepresentations of $\R^5$. Let $\alpha_1,\alpha_2$ be the standard integral basis of the dual Lie algebra of $T^2$. The following Figure\,\ref{fig:evenGrass} shows the GKM graphs of $G_2(\R^5)$ (\cite{GM14a} p.8 Example 3.8, \cite{He} p.12 (3A)) and $\tilde{G}_2(\R^5)$ (\cite{He} p.23 (3B)) under the canonical $T^2$-actions. To see the $\Z_2$-covering of GKM graphs: the dotted edge of $V_i$ in the GKM graph of $G_2(\R^5)$ corresponds to the folding of the solid edge that joins the pair $(V_{i_+},V_{i_-})$ in the GKM graph of $\tilde{G}_2(\R^5)$. 
	\begin{figure}[H]
		\centering
		\begin{subfigure}[b]{0.45\textwidth}
			\centering
			\begin{tikzpicture}[descr/.style={fill=white,inner sep=2.5pt},scale=1.5]
			\path (-1,0) node[draw,fill,circle,inner sep=0pt,minimum size=4pt,label=left:{$V_{1_+}$}] (v1+) {}
			(1,0) node[draw,fill,circle,inner sep=0pt,minimum size=4pt,label=right:{$V_{1_-}$}] (v1-) {}
			(0,1) node[draw,fill,circle,inner sep=0pt,minimum size=4pt,label=above:{$V_{2_+}$}] (v2+) {}
			(0,-1) node[draw,fill,circle,inner sep=0pt,minimum size=4pt,label=below:{$V_{2_-}$}] (v2-) {}
			(-0.05,0) node (x) {}
			(0.05,0) node (y) {}
			(0,-0.1) node (z) {}
			(0,0.1) node (w) {};
			
			\path[-] (v1+) edge node[left] {$\alpha_2-\alpha_1$} (v2+)
			(v1+) edge node[left] {$\alpha_2+\alpha_1$} (v2-)
			(v1-) edge node[right] {$\alpha_2-\alpha_1$} (v2-)
			(v1-) edge node[right] {$\alpha_2+\alpha_1$} (v2+)
			(v1+) edge (x)
			(v1-) edge node[descr] {$\alpha_1$} (y)
			(v2+) edge node[descr] {$\alpha_2$} (z)
			(w) edge (v2-);
			\end{tikzpicture}
			\caption{GKM graph of $\tilde{G}_{2}(\R^{5})$}
		\end{subfigure}
		\begin{subfigure}[b]{0.45\textwidth}
			\centering
			\begin{tikzpicture}[scale=1.5]
			\path (-1,0) node[draw,fill,circle,inner sep=0pt,minimum size=4pt,label={$V_1$} ] (v1){}
			(1,0) node[draw,fill,circle,inner sep=0pt,minimum size=4pt,label={$V_2$} ] (v2){}
			(-1.75,0) node[draw, fill, star, star points=5, star point ratio=2.25, inner sep=1pt] (v3){}
			(1.75,0) node[draw, fill, star, star points=5, star point ratio=2.25, inner sep=1pt] (v4){}
			(0,-1.3) node (v5){};
			
			\path[-] (v1) edge[bend right=30] node[below] {$\alpha_2-\alpha_1$} (v2)
			(v1) edge[bend left=30] node[above] {$\alpha_2+\alpha_1$} (v2);
			
			\draw[densely dotted] (v1) -- node[below] {$\alpha_1$} (v3)
			(v2) -- node[below] {$\alpha_2$} (v4);	
			\end{tikzpicture}
			\caption{GKM graph of $G_2(\R^5)$}
		\end{subfigure}
	\caption{A 2-cover between GKM graphs of the $T^2$-actions on the $6$d Grassmannians.}\label{fig:evenGrass}
	\end{figure}
\end{exm}
\noindent The details of the $1$-skeleta and the localization of the equivariant cohomology rings of the even-dimensional real and oriented Grassmannians can be found in \cite{He}. 

\vskip 20pt
\section{A GKM-type theorem in odd dimensions}
\vskip 15pt
With the even-dimensional GKM theory well established, it is natural to ask whether there is a parallel odd-dimensional analogue. Goertsches, Nozawa and T\"{o}ben \cite{GNT12} developed a GKM-type theory for a certain class of Cohen-Macaulay torus actions, including an application to certain K-contact manifolds. In this section, we introduce a GKM-type localization result for certain equivariantly formal torus actions on odd-dimensional, possibly non-orientable manifolds.

\subsection{GKM condition in odd dimensions}
As we have seen in the even-dimensional case, we need the 1-skeleton to be nonempty and of the smallest possible dimension. 

\begin{dfn}\label{OddGKMCond}
An action $T\curvearrowright M^{2n+1}$ is \textbf{odd GKM} (or has \textbf{minimal $1$-skeleton in odd dimension}) if it is equivariantly formal and the following is satisfied
\begin{itemize}
\item[(1)] The fixed-point set $M^T$ is a nonempty union of isolated circles. 
\item[(2)] The $1$-skeleton $M_1$ is 3-dimensional. Or equivalently, at any point $p$ on a fixed circle $\gamma\subset M^T$, the non-zero integral weights $\lambda_1,\ldots,\lambda_n \in \mathfrak{t}^*_\Z$ of the isotropy $T$-representation $T\curvearrowright T_{p}M\cong \R \oplus V_{\lambda_1}\oplus \cdots \oplus V_{\lambda_n}$ are pair-wise linearly independent. Using a continuity argument, one can show that the integral weights $\lambda_1,\ldots,\lambda_n$ depend on $\gamma$, not on the specific choice of $p\in \gamma$.
\end{itemize}
\end{dfn}

\begin{rmk}
 One of the anonymous referees suggested the following useful observation. Note $TM|_\gamma\cong T\gamma \oplus N\gamma \cong (\gamma \times \R)\oplus N\gamma$ where $N\gamma$ is the normal bundle of $\gamma$. Because a $\R^{2n}$-vector bundle over $S^1$ is determined by an element in $\pi_1(BO(2n))\cong \pi_0(O(2n)) \cong \Z_2$, there are only two types of $N\gamma$ up to bundle isomorphism, i.e. $N\gamma\cong S^1\times \R^{2n}$ or $(S^1\times_{\Z_2}\R)\times \R^{2n-1} $. If there is a global torus action on this vector bundle whose fixed-point set is the base space $S^1$, then $N\gamma$ must be the trivial bundle $\gamma \times \R^{2n}$. Since $\R^{2n}\cong V_{\lambda_1} \oplus \cdots \oplus V_{\lambda_n}$, then $N\gamma$ decomposes into $\gamma \times (V_{\lambda_1} \oplus \cdots \oplus V_{\lambda_n})$. Set $T_{\lambda_i}$ as the codimension-1 subtorus of $T$ with the Lie subalgebra $\mathrm{Ker}\,\lambda_i \subset \mathfrak{t}$, we have $TM^{T_{\lambda_i}}|_\gamma=\gamma \times \big(\R\oplus(\oplus_j V_{\lambda_j})\big)$ where the $\oplus_j$ is taken over all $\lambda_j$'s that are linear multiples of $\lambda_i$. This observation helps to see the equivalence between the condition that $M_1$ is 3d and the condition of pairwise independence of $\lambda_i$'s.
\end{rmk}

By Condition (1), the fixed-point set $M^{T}$ consists of circles. We fix a unit orientation form $\theta_{\gamma}$ for each circle $\gamma$, and write
\[
H^*_T(M^T)=\bigoplus_{\gamma\subset M^T}\big(H^*_T(pt)\otimes H^*(S^1_\gamma)\big)=\bigoplus_{\gamma\subset M^T}\big(\mathbb{S}\mathfrak{t}^*_\Q\oplus \mathbb{S}\mathfrak{t}^*_\Q\theta_\gamma\big).
\] 

By Condition (2), for any isotropy weight $\lambda$ along a fixed circle $\gamma\subset M^T$, the component $C_{\lambda}$ of $M^{T_{\lambda}}$ containing $\gamma$ is $3$-dimensional and has an effective residual action of the circle $T/T_{\lambda}$. 

\subsection{Geometry and equivariant cohomology of 3d $S^1$-manifolds}
3-dimensional $S^1$-manifolds without fixed points are the Seifert manifolds. The case of 3-dimensional $S^1$-manifolds with or without fixed points, also called generalized Seifert manifolds, was classified by Orlik and Raymond.

Briefly speaking, the equivariant diffeomorphism type of a 3-dimensional $S^1$-manifold $M$ is determined by the orbifold type of the quotient space $M/{S^1}$, the numerical data of the Seifert fibres over the orbifold points of $M/{S^1}$, and the orbifold Euler number $b$ of the ``fibration'' $M \rightarrow M/{S^1}$.  

Let us write $\epsilon=o$ (orientable) or $n$ (non-orientable) for the orientability of the orbifold surface $M/{S^1}$, and $g$ for its genus. Write $f$ for the number of connected components in the fixed-point set $M^{S^1}=\cup_{i=1}^f \gamma_i$, write $s$ for the number of connected components of special exceptional orbits whose normal spaces viewed as isotropy representations are $\mathbb{Z}_2 \overset{\text{reflect}}{\curvearrowright} \R^2$ such that $\exp(\pi \sqrt{-1})\cdot (x,y)=(x,-y)$ for any $(x,y)\in \R^2$, and record a coprime integer pair $(\mu_j,\,\nu_j)$ for each Seifert fiber in $M^{\mathbb{Z}/\mu_j}$ whose $2$-dimensional normal space viewed as an isotropy representation is $\mathbb{Z}_{\mu_j} \overset{\text{rotate}}{\curvearrowright} \R^2$ such that $\exp(\frac{2\pi \sqrt{-1}}{\mu_j}) \cdot z = \exp(\frac{2\pi \nu_j \sqrt{-1}}{\mu_j}) z$ for any $z\in \R^2\cong \C$.

The non-orientability of $M$ is equivalent to either $\epsilon =n$ (the quotient $M/S^1$ is non-orientable) or $s>0$ (there are special exceptional orbits). 

All the above mentioned data together gives a numerical criterion of classifying $3$-dimensional $S^1$-manifolds.

\begin{thm}[Orlik-Raymond classification of closed 3d $S^1$-manifolds, \cite{Ra68,OR68}]
Let $S^1$ act effectively and smoothly on a closed, connected smooth $3$-dimensional manifold $M$. Then the orbit invariants
\[
\big\{b;(\epsilon,g,f,s);(\mu_1,\,\nu_1),\ldots,(\mu_r,\,\nu_r)\big\}
\]
subject to certain conditions, determine $M$ up to equivariant diffeomorphism.
\end{thm} 

The proof of this theorem is by equivariant cutting and pasting, and furthermore inspires us to compute the equivariant cohomology using Mayer-Vietoris sequences and classify equivariantly formal $S^1$-actions on $3$-dimensional manifolds.

\begin{thm}[Equivariantly formal 3d $S^1$-manifold, \cite{He17} p.\,258 Thm 4.8]\label{3dFormal}
A closed $3$-dimensional $S^1$-manifold $M = \big\{b;(\epsilon,g,f,s);(\mu_1,\,\nu_1),\ldots,(\mu_r,\,\nu_r)\big\}$ is $S^1$-equivariantly formal if and only if $f>0,\,b=0$ and one of the following three constraints holds
\[
\begin{cases}
\hfill \epsilon=o,\,g=0,\,s=0~\\
\hfill \epsilon=o,\,g=0,\,s=1~\\
\hfill \epsilon=n,\,g=1,\,s=0.\\
\end{cases}
\]
\end{thm} 

In the author's proof of the above theorem on equivariantly formal $3$-dimensional $S^1$-manifold, it was shown that the elements of the equivariant cohomology $H_{S^1}^*(M)$ have the following expression when localized to $H^*_{S^1}(M^{S^1})=\bigoplus_{i=1}^f (H^*_{S^1}(pt)\otimes H^*(\gamma_i))=\bigoplus_{i=1}^f (\mathbb{Q}[\alpha]\otimes H^*(\gamma_i))$:
\[
\sum_{i=1}^f\big(P_i(\alpha)\delta_i+Q_i(\alpha)\theta_i\big) \in \bigoplus_{i=1}^f \big(\mathbb{Q}[\alpha]\otimes H^*(\gamma_i)\big)
\]
where $P_i,\,Q_i \in \mathbb{Q}[\alpha]$ are polynomials and $\delta_i,\theta_i$ are generators of $H^0(\gamma_i,\Z),H^1(\gamma_i,\Z)$. Moreover,
\begin{enumerate}
	\item in the case of $\epsilon=o,\,g=0,\,s=0$ such that the 3d manifold $M$ is orientable and the $S^1$-action preserves an orientation (we will call this case \textbf{$S^1$-orientable}), those $P_i,Q_i$ are subject to two relations:
	\[
	P_1(0)=P_2(0)=\cdots=P_f(0) \mbox{ and } \sum_{i=1}^f Q_i(0)=0
	\]
	where the second one is obtained under the assumption that after fixing an orientation on $M$, the $\theta_i$'s represent the induced orientations on $\gamma_i$'s (see Subsubsection \ref{subsub:orientation}).
	\item in the two cases of $\epsilon=o,\,g=0,\,s=1$ and $\epsilon=n,\,g=1,\,s=0$ such that $M$ is either non-orientable, or orientable but the $S^1$-action does not preserve an orientation (we will call these cases \textbf{non-$S^1$-orientable}), those $P_i,Q_i$ are subject to the relation:
	\[
	P_1(0)=P_2(0)=\cdots=P_f(0).
	\]
\end{enumerate}

Back to an odd GKM action $T\curvearrowright M^{2n+1}$, let $C_\lambda$ be a component of weight $\lambda$ in the 1-skeleton $M_1$. The codimension-$1$ subtorus $T_\lambda$ acts on $C_\lambda$ trivially and the residual circle $T/T_\lambda$ acts equivariantly formally. The elements of the equivariant cohomology $H^*_T(C_\lambda)$ can be localized to the fixed-point set $\cup_{i=1}^f \gamma_i$ in the form:
\[
\sum_{i=1}^f\big(P_i\delta_i+Q_i\theta_i\big) \in \bigoplus_{i=1}^f \big(\mathbb{S}\mathfrak{t}^*_\Q\otimes H^*(\gamma_i)\big)
\]
where $\delta_i,\theta_i$ are generators of $H^0(\gamma_i,\Z),H^1(\gamma_i,\Z)$. Since $T_\lambda$ acts trivially on $C_\lambda$, we have $H^*_T(C_\lambda)\cong H^*_{T/T_\lambda}(C_\lambda)\otimes H^*_{T_\lambda}(pt)$:
\begin{enumerate}
	\item when $C_\lambda$ is $T/T_\lambda$-oriented and $\theta_i$'s represent the induced orientations on $\gamma_i$'s, the polynomials $P_i,\,Q_i \in \mathbb{S}\mathfrak{t}^*_\Q$ are subject to two congruence relations:
	\begin{equation}\label{Relations1}
	P_1\equiv P_2\equiv \cdots\equiv P_f \mbox{ and } \sum_{i=1}^f Q_i\equiv 0 \mod{\lambda} \tag{$\dag$};
	\end{equation}
	\item when $C_\lambda$ is non-$T/T_\lambda$-orientable, the polynomials $P_i,\,Q_i \in \mathbb{S}\mathfrak{t}^*_\Q$ are subject to the congruence relation:
	\begin{equation*}\label{Relations1'}
	P_1\equiv P_2\equiv \cdots\equiv P_f \mod{\lambda} \tag{$\ddag$}. 
	\end{equation*}
\end{enumerate}

In the next three subsubsections, firstly we show that the vast possibilities of $S^1$-equivariantly formal, closed 3d manifolds can be considerably reduced if we impose extra structures, then we give a method of constructing all the 3d closed, equivariantly formal, $S^1$-manifolds.

\subsubsection{Mapping tori of 2d $S^1$-manifolds} \label{subsubsec:MapTori}
Up to $S^1$-diffeomorphisms, (see \cite{Au04} p.\,18-20) there are two orientable 2-dimensional $S^1$-manifolds: the rank-2 torus $T^2=S^1\times S^1$ with $S^1$ acting on the first factor, the 2-sphere $S^2$ with the standard $S^1$-action; their non-orientable $\Z_2$-quotients: the Klein bottle $K=S^1\times S^1/(z_1,z_2)\sim (-z_1,\bar{z}_2)$ and the real projective plane $\R P^2 = S^2/\Z_2$ with the induced $S^1$-actions. 

Let $N$ be one of the four 2d $S^1$-manifolds, and $\phi$ be an $S^1$-automorphism on $N$. We form the $S^1$-equivariant mapping torus
\[
N_\phi \triangleq \frac{N\times [0,1]}{N\times\{0\}\sim_\phi N\times\{1\}}
\]
whose $S^1$-diffeomorphism type is determined by the $S^1$-isotopy type of $\phi$. 

For $N=T^2$, an $S^1$-automorphism $\phi$ induces an automorphism on the orbit space: $\phi/S^1: \{1\}\times S^1\rightarrow \{1\}\times S^1$ which is isotopic to the identity map $z\mapsto z$ or the inverse map $z\mapsto z^{-1}$ where $z\in S^1 \subset \C$. In addition to $\phi/S^1$, the $S^1$-isotopy type of $\phi$ is determined by the number of times that $\phi(\{1\}\times S^1)$ wraps along the direction of $S^1 \times \{1\}$, i.e. the mapping degree of $pr_1\circ \phi:\{1\}\times S^1 \rightarrow S^1 \times \{1\}$ where $pr_1$ projects $T^2$ to its first factor. Up to $S^1$-isotopy, we have $\phi:T^2\rightarrow T^2:(z_1,z_2)\mapsto (z_1 z^k_2,z_2)$ or $(z_1 z^k_2,z^{-1}_2)$ for some integer $k$.  However, those mapping tori $T^2_\phi$ do not have $S^1$-fixed points.

For $N=S^2,\,K,\,\R P^2$, the orbit space $N/S^1$ is an interval $[0,1]$. There is a cross section for the projection $N\rightarrow N/S^1=[0,1]$, hence the $S^1$-isotopy type of $\phi:N\rightarrow N$ is determined by the isotopy type of $\phi/S^1:[0,1]\rightarrow [0,1]$. If $N=S^2,\,K$, the two end points of $N/S^1=[0,1]$ represent two $S^1$-fixed points or two $S^1/\Z_2$ orbits respectively. Therefore, up to isotopy, the automorphism $\phi/S^1:[0,1]\rightarrow [0,1]$ is either the identity $t\mapsto t$ or $t\mapsto 1-t$ swapping the two end pints. If $N=\R P^2$, one end point of $\R P^2/S^1=[0,1]$ represents an $S^1$-fixed point and the other one represents an $S^1/\Z_2$ orbit. Therefore, up to isotopy, the automorphism $\phi/S^1:[0,1]\rightarrow [0,1]$ is the identity map. 

If a mapping torus $N_\phi$ is $S^1$-equivariantly formal, then it must have $S^1$-fixed points, hence $N=S^2,\,\R P^2$. For $N=S^2$ and $\phi=id$, we have $N_\phi=S^2\times S^1$. This is a 3d $S^1$-orientable, equivariantly formal manifold  $\big(b=0;(\epsilon=o,g=0,f=2,s=0)\big)$ with two $S^1$-fixed circles. We have
\[H^*_{S^1}(S^2\times S^1)\cong H^*_{S^1}(S^2)\otimes H^*(S^1)\cong\big\{ \big(P_N(\alpha),Q_N(\alpha) \theta; P_S(\alpha),Q_S(\alpha)\theta\big)\mid P_N(0)=P_S(0),\,Q_N(0)=Q_S(0)\big\}.\] 
For $N=S^2$ and $\phi$ being the antipodal map that swaps the north pole with the south pole, we have $N_\phi=S^2\times_{\Z_2} S^1$. This is a 3d non-$S^1$-orientable, equivariantly formal manifold $\big(b=0;(\epsilon=n,g=1,f=1,s=0)\big)$ with one $S^1$-fixed circle. We have
\[
H^*_{S^1}(S^2\times_{\Z_2} S^1)\cong \big\{ \big(P(\alpha),Q(\alpha) \theta\big) \big\} \cong H^*_{S^1}(pt)\otimes H^*(S^1).
\] 
For $N=\R P^2$ and $\phi=id$, we have $N_\phi=\R P^2\times S^1$. This is a 3d non-$S^1$-orientable, equivariantly formal manifold $\big(b=0;(\epsilon=o,g=0,f=1,s=1)\big)$ with one $S^1$-fixed circle. We have
\[
H^*_{S^1}(\R P^2\times S^1)\cong H^*_{S^1}(\R P^2)\otimes H^*(S^1) \cong \big\{ \big(P(\alpha),Q(\alpha) \theta\big) \big\} \cong H^*_{S^1}(pt)\otimes H^*(S^1).
\]

\subsubsection{Extendible $S^1$-actions on 3d manifolds}\label{subsubsec:ExtendibleAction}
We say an effective action $S^1\curvearrowright M^3$ is \textbf{extendible} if it is a subaction of an effective action $G\curvearrowright M^3$ where $G$ is a compact connected Lie group properly containing $S^1$. If $S^1$ is a maximal torus of $G$, then $G$ is an $SU(2)$ or $SO(3)$. Otherwise, $S^1$ is contained as the first factor $S^1\times \{1\}$ of a rank-2 subtorus $T^2\subseteq G$. Now we assume the extended group $G$ is one of $SU(2),\,SO(3),\,T^2$ and note that $\dim M^3/G<\dim M^3/S^1=2$. 

If $\dim M^3/G = 0$, then $G$ is $SU(2)\cong S^3$ or $SO(3)\cong \R P^3 \cong S^3/\Z_2$, hence $M$ is $S^3/\Gamma$ for a finite subgroup $\Gamma \subset SU(2)$. On one hand, $M$ is a rational homology $3$-sphere. On the other hand, the $S^1$-action on $M$ is induced from $SU(2)\cong S^3$ which has at least an $S^1$-fixed circle. Then $\dim H^*(M)=2$ and $\dim H^*(M^{S^1})\geq 2$. By the Formality Criterion \ref{CohomIneq}, the $S^1$-action on $M$ is equivariantly formal. We have the $H^*_{S^1}(pt)$-module isomorphisms 
$$H^*_{S^1}(M) \cong H^*_{S^1}(pt)\otimes H^*(M) \cong H^*_{S^1}(pt)\otimes H^*(S^3)$$ which are also $H^*_{S^1}(pt)$-algebra isomorphisms because the degree-3 generator must have zero square. The localized expression is
\[
H^*_{S^1}(M) \cong H^*_{S^1}(S^3)\cong \big\{ \big(P(\alpha),Q(\alpha) \theta\big) \mid Q(0)=0\big\}.
\]

If $\dim M^3/G = 1$, Neumann (\cite{Ne68} p.\,221) showed that $G$ is $SO(3)$ or $T^2$. The types S4-10 of Neumann's classification have $S^1$-fixed points and can be verified to be $S^1$-equivariantly formal using Theorem\,\ref{CohomIneq}\,or\,\ref{3dFormal}:
\begin{itemize}
	\item If $G$ is $T^2$, we have the types $S4=\R P^2 \times S^1,\,S5=S^2\times_{\Z_2} S^1,\,S6=S^2\times S^1$ whose $S^1$-equivariant cohomology has been given in the previous discussion about mapping tori. For coprime integers $0<q<p$, the $3$d lens spaces $S7=L(p,q)\cong S^3/\Z_p$ (including $S^3=L(1,0)$), have the $S^1$-equivariant cohomology
	\[
	H^*_{S^1}(L(p,q))\cong H^*_{S^1}(S^3) \cong \big\{ \big(P(\alpha),Q(\alpha) \theta\big) \mid Q(0)=0\big\}.
	\]
	\item If $G$ is $SO(3)$, the types $S8=\R P^2 \times S^1,\,S9'=S^2\times_{\Z_2} S^1,\,S9=S^2\times S^1$ have appeared as the mapping tori of 2d $SO(3)$-manifolds. The types $S10=\R P^3 \# \R P^3,\,S11= S^3,\, S12=\R P^3$ are rational homology $3$-spheres and have the $S^1$-equivariant cohomology
	\[
	H^*_{S^1}(\R P^3 \# \R P^3)\cong H^*_{S^1}(\R P^3)\cong H^*_{S^1}(S^3) \cong \big\{ \big(P(\alpha),Q(\alpha) \theta\big) \mid Q(0)=0\big\}.
	\]
\end{itemize}

\subsubsection{Constructing the 3d closed equivariantly formal $S^1$-manifolds}\label{subsub:constuct3dFromal}
Using connected sums, Raymond gave a construction of all 3d closed $S^1$-manifolds that have non-empty fixed-point sets. In particular, we can apply his results to construct the 3d equivariantly formal $S^1$-manifolds. By \cite[p.\,58-59, Thm 1.(ii)]{Ra68}, the orbit data $\big\{b=0;(\epsilon=o,g=0,f>0,s=0)\big\}$, $\big\{b=0;(\epsilon=o,g=0,f>0,s=1)\big\}$ and $\big\{b=0;(\epsilon=n,g=1,f>0,s=0)\big\}$ respectively can be realized as
\[
S^3\#\big(\#_{f-1} (S^2\times S^1)\big)
\qquad
S^3\#\big(\#_{f-1} (S^2\times S^1)\big)\#(\R P^2\times S^1)
\qquad
(S^2\times_{\Z_2}S^1)\#\big(\#_{f-1} (S^2\times S^1)\big)
\]
where $\#_{f-1}$ means taking connected sum of $f-1$ copies. Let $M'$ be one of the above three types. By \cite[p.\,72, Thm 4]{Ra68}, we can additionally realize the orbit invariants $(\mu_i,\nu_i)$ by taking connected sum with a 3d Lens space $L(\mu_i,\nu_i)$ with an $S^1$-action given in \cite[p.\,70-71, Sec.\,7]{Ra68}. All the 3d closed equivariantly formal $S^1$-manifolds satisfying the condition in Theorem \ref{3dFormal} can be constructed as
\[
M'\#\big(\#_{i=1}^r L(\mu_i,\nu_i)\big).
\]
For details of equivariant connected sum, see \cite[p.\,71-72, Sec.\,8]{Ra68}.

Another idea of constructing the 3d equivariantly formal $S^1$-manifolds was suggested by one of the anonymous referees. We elaborate on that idea and give the details as follows.

Let $M_1$, $M_2$ be 3d closed $S^1$-manifolds such that one of them has a non-empty fixed-point set, we will define a new 3d closed $S^1$-manifold $M_1\natural M_2$ with a non-empty fixed-point set. First, we delete an $S^1$-invariant
neighbourhood of a free $S^1$-orbit from each of $M_i$. Such a neighbourhood is $S^1$-diffeomorphic to a solid torus $S^1\times D^2$ with an $S^1$-action concentrating on the $S^1$-factor. Then, we glue $M_i\setminus (S^1\times D^2),\,i=1,2$ along the boundary tori $S^1\times S^1$ via an $S^1$-equivariant automorphism $\varphi$ on $S^1\times S^1$, and denote the glued manifold by $M_1\natural_\varphi M_2$. If $M_i$'s are both $S^1$-orientable, the automorphism has to be orientation-reversing in order to make the glued manifold orientable. As we observed in Subsubsection \ref{subsubsec:MapTori}, the $S^1$-equivariant automorphism $\varphi$ on $S^1\times S^1$ is not unique. Note that, at the level of orbit space, we have $$(M_1\natural_\varphi M_2)/S^1=(M_1/S^1) \# (M_2/S^1).$$ Using Raymond and Orlik's local analysis of $S^1$-manifolds \cite{Ra68,OR68}, with different $\varphi$, the orbit data of the glued manifolds $M_1\natural_\varphi M_2$ will at most differ on the Euler number $b$. Also note that $$(M_1\natural_\varphi M_2)^{S^1}=M_1^{S^1} \cup M_2^{S^1}.$$If one of $M_i$ has non-empty fixed-point set, then that Euler number $b$ of $M_1\natural_\varphi M_2$ vanishes \cite[p.\,69, Cor 2a]{Ra68}. Hence the $S^1$-diffeomorphism type of $M_1\natural_\varphi M_2$ is independent of $\varphi$, and we will denote it by $M_1\natural M_2$.

Let $S^3\subset \C^2$ be induced with the standard $S^1$-action on the first coordinate. Then $S^3$ realizes the orbit data $\big\{b=0;(\epsilon=o,g=0,f=1,s=0)\big\}$. In Subsubsection \ref{subsubsec:MapTori}, we have seen that $\R P^2\times S^1$ and $S^2 \times_{\Z_2} S^1$ with $S^1$-actions on the first factors respectively realize the orbit data $\big\{b=0;(\epsilon=o,g=0,f=1,s=1)\big\}$, $\big\{b=0;(\epsilon=n,g=1,f=1,s=0)\big\}$. Using the previous observations $(M_1\natural M_2)/S^1=(M_1/S^1) \# (M_2/S^1)$, $(M_1\natural M_2)^{S^1}=M_1^{S^1} \cup M_2^{S^1}$, we can realize the orbit data $\big\{b=0;(\epsilon=o,g=0,f>0,s=0)\big\}$, $\big\{b=0;(\epsilon=o,g=0,f>0,s=1)\big\}$ and $\big\{b=0;(\epsilon=n,g=1,f>0,s=0)\big\}$ respectively as
\[
\natural_f S^3 \qquad\qquad (\R P^2\times S^1)\natural (\natural_{f-1} S^3) \qquad\qquad  (S^2 \times_{\Z_2} S^1)\natural (\natural_{f-1} S^3)
\]
where $\natural_{f-1}$ means taking the $\natural$ construction of $f-1$ copies. Let $M'$ be one of the above three types. Let $M''$ be a Seifert manifold that realizes the orbit data $\big\{b=0;(\epsilon=o,g=0,f=0,s=0);(\mu_1,\,\nu_1),\ldots,(\mu_r,\,\nu_r)\big\}$. A 3d closed equivariantly formal $S^1$-manifolds satisfying the condition in Theorem \ref{3dFormal} can also be constructed as
\[
M'\natural M''.
\]

\subsection{GKM graphs and a GKM-type theorem in odd dimensions}
We will construct GKM graphs for odd-dimensional GKM manifolds and give a graph-theoretic computation of their equivariant cohomology rings. 

In the even-dimensional case, each $S^2$ or $\R P^2$ in the 1-skeleton gives an edge connecting two vertices in a GKM graph. However, in the odd-dimensional case, as we have seen in the previous discussions, a component in the 1-skeleton could contain any positive number of fixed circles, in contrast to the exactly two fixed points of $S^2$. Due to this difference, the construction of the graphs in odd dimensions will be slightly more complicated.

\begin{dfn}\label{dfn:oddGKMGraph}
	The \textbf{odd GKM} (\textbf{$1$-skeleton}) graph for an odd GKM action $T\curvearrowright M^{2n+1}$ consists of: 
	\begin{description}
		\item[Vertices$\,\&\,$Weights] There are two types of vertices:
		\begin{itemize}
			\item [$\circ$] for each fixed circle $\gamma \subset M^T$,
			\item [$\Box$] for each $3$-dimensional component $C$ in $M^{T_\lambda}$ of some codimension-1 subtorus $T_\lambda$. The $\Box$ is weighted by $(\lambda,\varepsilon)$ where $\varepsilon=O$ (orientable) or $N$ (non-orientable) for the $T/T_\lambda$-orientability of $C$. If the $T/T_\lambda$-orientability of $C$ is known in the context, then we might drop the symbol $\varepsilon$.
		\end{itemize}
		
		\item[Edges] If a $3$-dimensional component $C$ contains a fixed circle $\gamma$, then an edge joins a $\Box$ that represents $C$ to a $\circ$ that represents $\gamma$. No edges directly join $\circ$ to $\circ$, nor $\Box$ to $\Box$.
		
	\end{description}
\end{dfn}

\begin{rmk}
	We point out some comparisons between the even GKM graphs and odd GKM graphs.
	\begin{itemize}
		\item In contrast to labelling weights on the edges of an even GKM graph, we label weights on the $\Box$ vertices of an odd GKM graph. The seemly difference is actually in the same spirit, because the weights are associated to the components of the $1$-skeleton, which are represented by edges in an even GKM graph while by $\Box$ vertices in an odd GKM graph.
		\item By the odd-dimensional GKM condition \ref{OddGKMCond}, a fixed circle has exactly $n$ pair-wise independent weights. Thus each $\circ$, representing a fixed circle, is joined via exactly $n$ edges to $n$ $\Box$'s. A 3-dimensional component $C\subseteq M_1$ can contain any positive number of fixed circles, thus a $\Box$ can be joined via any positive number of edges to $\circ$'s. 
	\end{itemize}
\end{rmk}
\begin{exm}
	All the $3$d closed equivariantly formal $S^1$-manifolds, that we used in Theorem \ref{3dFormal}, are the building blocks of the odd-dimensional GKM theory. See Subsubsection \ref{subsub:constuct3dFromal} for the two constructions of the $3$d equivariantly formal $S^1$-manifolds. For the equivariantly formal, $S^1$-orientable manifolds $M^3=\big\{b=0;\epsilon=o,g=0,f>0,s=0;(\mu_1,\,\nu_1),\ldots,(\mu_r,\,\nu_r)\big\}$, the odd GKM graph (Figure \ref{fig:3d}) consists of a unique $\Box$-vertex of weight $\alpha$, and $f$ edges joining that one $\Box$-vertex with $f$ $\circ$-vertices. It is worth noting that the existence of the invariants $(\mu_i,\,\nu_i)$ does not affect the odd GKM graph nor the rational equivariant cohomology.
	\begin{figure}[H]
		\centering
		\begin{tikzpicture}[scale=1.5]
		\path (0:0cm) node[draw,rectangle,scale=1.2,label=left:{$(C,\alpha)$}] (v0){};
		\path (0:1cm) node[draw,circle,label={$\gamma_{_1}$}] (v1){};
		\path (72:1cm) node[draw,circle,label=left:{$\gamma_{_2}$}] (v2){};
		\path (3*72:1cm) node[draw,circle,label=left:{$\gamma_{_{f-1}}$}] (v4){};
		\path (4*72:1cm) node[draw,circle,label=right:{$\gamma_{_{f}}$}] (v5){};
		
		\draw (v0) -- (v1)
		(v0) -- (v2)
		(v0) -- (v4)
		(v0) -- (v5);
		\draw[dashed] (2*50:0.85cm) arc (2*50:2*100:0.85cm);
		\end{tikzpicture}
		\caption{Odd GKM Graph of a 3d equivariantly formal $S^1$-manifold.}\label{fig:3d}
	\end{figure}
	\noindent The equivariant cohomology is $H^*_{S^1}(M)=\big\{(P_1,Q_1\theta;\ldots;P_f,Q_f\theta) \in (\mathbb{Q}[\alpha] \oplus \mathbb{Q}[\alpha]\theta)^{\oplus f}  \mid P_1(0)=\cdots=P_f(0),\, \sum_{i=1}^f Q_i(0)=0\big\}$. For the equivariantly formal, non-$S^1$-orientable manifolds $M^3=\big\{b=0;\epsilon=o,g=0,f>0,s=1;(\mu_1,\,\nu_1),\ldots,(\mu_r,\,\nu_r)\big\} \text{ or } \big\{b=0;\epsilon=n,g=1,f>0,s=0;(\mu_1,\,\nu_1),\ldots,(\mu_r,\,\nu_r)\big\}$, the odd GKM graph looks the same as the oriented case, but the $\Box$ represents an un-orientable 3-dimensional manifold, whose equivariant cohomology is $H^*_{S^1}(M)=\big\{(P_1,Q_1\theta;\ldots;P_f,Q_f\theta) \in (\mathbb{Q}[\alpha] \oplus \mathbb{Q}[\alpha]\theta)^{\oplus f}  \mid P_1(0)=\cdots=P_f(0)\big\}$.
	
\end{exm}

Let us describe a GKM-type theorem for the equivariant cohomology $H^*_T(M^{2n+1})$ in a graph-theoretic way. First, if a $3$-dimensional component $C \subseteq M_1$ is orientable, then we choose its orientation in advance. We also choose an orientation $\theta_{\gamma_i}$ for each fixed circle $\gamma_i \subseteq M^T$. We drop the subscript of $\theta_{\gamma_i}$ and simply write $\theta$ universally for every $\gamma_i$. 

\begin{dfn}
	Let $T\curvearrowright M^{2n+1}$ be an odd GKM action with the odd GKM graph $\mathcal{G}$ consisting of two types of vertex sets $V_\circ$ and $V_\Box$ and the edge set $E$. The \textbf{cohomology of the odd GKM graph $\mathcal{G}$}, denoted by $H^*(\mathcal{G})$, is the set of the following paired maps:
	\[
	(P,Q\theta): V_\circ \longrightarrow \mathbb{S}\mathfrak{t}^*_\Q \oplus \mathbb{S}\mathfrak{t}^*_\Q \theta
	\]
	where $\theta$ is a generator of $H^1(S^1,\Z)$, and $P,Q$ satisfy the following congruence relations contributed from each $(\Box,\lambda,\varepsilon)$ representing a $3$-dimensional component $C$ of $M^{T_{\lambda}}$, and its neighbour $\circ$'s representing the fixed circles $\gamma_1,\ldots,\gamma_{k}$ on $C$:
	\begin{itemize}
		\item if $\varepsilon=O$, i.e. $C$ is $T/T_{\lambda}$-oriented, we have
		\begin{equation*}
		P_{\gamma_1}\equiv P_{\gamma_2}\equiv \cdots\equiv P_{\gamma_{k}} \mbox{ and } \sum_{i=1}^k \pm Q_{\gamma_i}\equiv 0 \mod{\lambda} 
		\end{equation*}
		where the sign for each $Q_{\gamma_i}$ is specified by comparing the prechosen orientation $\theta_i$ with the induced orientation of $C$ on $\gamma_i$ (see Subsubsection \ref{subsub:orientation});
		\item if $\varepsilon=N$, i.e. $C$ is non-$T/T_{\lambda}$-orientable, we have 
		\begin{equation*}
		P_{\gamma_1}\equiv P_{\gamma_2}\equiv \cdots\equiv P_{\gamma_{k}} \mod{\lambda}. 
		\end{equation*}
	\end{itemize}
\end{dfn}

\begin{rmk}
	If we reverse the prechosen orientations of the fixed circles $\gamma_i$ or of the $3$-dimensional orientable components $C$ of $M^{T_\lambda}$, then we change the signs in front of $Q_{\gamma_i}$ accordingly. 
\end{rmk}

\begin{rmk}
	The $\mathbb{S}\mathfrak{t}^*_\Q$-algebra structure on $H^*(\mathcal{G})$ is canonically induced from $\mathbb{S}\mathfrak{t}^*_\Q \oplus \mathbb{S}\mathfrak{t}^*_\Q \theta$. Write an element $(P,Q\theta)$ as $(P_\gamma+Q_\gamma\theta)_{\gamma\subset M^T}$. We have $(P_\gamma+Q_\gamma\theta)_{\gamma\subset M^T} + (\bar{P}_\gamma+\bar{Q}_\gamma\theta)_{\gamma\subset M^T}=([P_\gamma+\bar{P}_\gamma]+[Q_\gamma+\bar{Q}_\gamma]\theta)_{\gamma\subset M^T}$. Note $\theta^2=0$, then $(P_\gamma+Q_\gamma\theta)_{\gamma\subset M^T} \cdot (\bar{P}_\gamma+\bar{Q}_\gamma\theta)_{\gamma\subset M^T}=([P_\gamma \bar{P}_\gamma]+[P_\gamma\bar{Q}_\gamma+\bar{P}_\gamma Q_\gamma]\theta)_{\gamma\subset M^T}$. The $\mathbb{S}\mathfrak{t}^*_\Q$-module structure of $H^*_T(M)$ is that, for any polynomial $R \in \mathbb{S}\mathfrak{t}^*_\Q$, we have $R \cdot (P_\gamma+Q_\gamma\theta)_{\gamma\subset M^T} = (R P_\gamma+R Q_\gamma\theta)_{\gamma\subset M^T}$.
\end{rmk}

\begin{thm}[A GKM-type theorem in odd dimensions]\label{thm:OddGKM}
	If an effective $T$-action on an odd-dimensional, possibly non-orientable manifold $M^{2n+1}$ is odd GKM with the graph $\mathcal{G}$, then there is an $\mathbb{S}\mathfrak{t}^*_\Q$-algebra isomorphism:
	\[
	H^*_T(M)\cong H^*(\mathcal{G}).
	\]
\end{thm}

\begin{proof}
	The odd-dimensional GKM condition \ref{OddGKMCond} assumes that the fixed-point set $M^T$ is a non-empty union of isolated circles, and that the $1$-skeleton $M_1$ is a union of $3$-dimensional manifolds with residual circle actions.
	
	An element of $H^*_T(M^T)$ can be written as follows: to each fixed circle $\gamma\subset M^T$, we associate a pair of polynomials  $(P_\gamma\delta_\gamma, Q_\gamma \theta_\gamma)\in \mathbb{S}\mathfrak{t}^*_\Q\otimes H^*(\gamma)$, where $\delta _\gamma,\theta_\gamma$ are the generators of $H^0(\gamma,\Z),H^1(\gamma,\Z)$. Equivalently, we have a paired map $(P,Q\theta): V_\circ \longrightarrow \mathbb{S}\mathfrak{t}^*_\Q \oplus \mathbb{S}\mathfrak{t}^*_\Q \theta$.
	
	By Proposition \ref{InheritFormality} on the inheritance of equivariant formality, the residual action of the circle $T/T_\lambda$ on a $3$-dimensional component $C\subset M^{T_\lambda}$, represented by a $\Box \in V_\Box$, is also equivariantly formal. Then we can use the Classification Theorem \ref{3dFormal} of equivariantly formal $S^1$-actions on closed $3$-dimensional manifolds, and the congruence relations \eqref{Relations1}, \eqref{Relations1'} therein to describe the embedded image $\mathrm{Im}\big( H^*_T(C) \rightarrow  H^*_T(M^T)\big)$.
	
	The only modifications are the signs in $\sum_{i=1}^k \pm Q_{\gamma_i}$. Notice that in Theorem \ref{3dFormal}, when $C$ is oriented, the orientation forms $\theta_{\gamma_i}$ are induced from $C$, such that the isotropy weight at each $\gamma_i$ is equal to $1$ under the effective residual action of the circle $T/T_\lambda$, or equivalently its isotropy weight under the $T$-action is $\lambda$. However, if we have chosen orientations in advance for the fixed circles $\gamma_i$, then we need to adjust the signs of $Q_{\gamma_i}$ in relation \eqref{Relations1} for the difference of the prechosen orientations and the induced orientations.
	
	The set of paired maps $(P,Q\theta): V_\circ \longrightarrow \mathbb{S}\mathfrak{t}^*_\Q \oplus \mathbb{S}\mathfrak{t}^*_\Q \theta$ that satisfy all the specified congruence relations is exactly the set 
	$$ \bigcap_{\lambda} \Big( \mathrm{Im}\big( H^*_T(M^{T_\lambda}) \rightarrow  H^*_T(M^T)\big)\Big)$$
	where the intersection is taken over all (finitely many) isotropy weights of the action $T\curvearrowright M^{2n+1}$. By the Chang-Skjelbred Lemma \ref{Chang}, we get the full description of $H^*_T(M)$.
\end{proof}

\vskip 20pt
\section{Examples}
\vskip 15pt

In this section, we construct some odd-dimensional GKM manifolds with certain additional geometric or topological structures and apply the odd GKM Theorem \ref{thm:OddGKM} to describe the equivariant cohomology with graphs.

\subsection{Product of an odd GKM manifold and an even GKM manifold}
Given an even GKM action $T^k \curvearrowright M^{2m}$ and an odd GKM action $T^l \curvearrowright N^{2n+1}$, we get the product action $T^k \times T^l \curvearrowright M^{2m}\times N^{2n+1}$ which is odd GKM. Using $T^k \times T^l\cong T^{k+l}$, we identify the weights of $T^k$ and $T^l$ as weights of $T^{k+l}$. The odd GKM graph $\mathcal{G}_{M\times N}$ can be obtained from the even GKM graph $\mathcal{G}_M$ together with the odd GKM graph $\mathcal{G}_N$. First, for each vertex $p$ of $\mathcal{G}_M$, we place a copy of $\mathcal{G}_N$ and denote it by $\mathcal{G}_N^p$. Second, for each solid edge $\overline{pq}$ of $\mathcal{G}_M$ that represents an $S^2$ of weight $\lambda$, we take each $\circ$-vertex of $\mathcal{G}_N^p$ and its corresponding $\circ$-vertex of $\mathcal{G}_N^q$, then join these two $\circ$-vertices via a new $(\Box,\lambda)$ that represents an $S^2\times S^1$ with the residual $S^1$-action concentrated on the $S^2$-factor. Last, for each dotted edge at a vertex $p$ of $\mathcal{G}_M$ that represents a $\R P^2$ of weight $\mu$, we join each $\circ$-vertex of $\mathcal{G}_N^p$ to a new $(\Box,\mu)$ that represents a $\R P^2\times S^1$ with the residual $S^1$-action concentrated on the $\R P^2$-factor.

\begin{rmk}
	More generally, it is possible to describe the GKM graphs of GKM bundles with even dimensional GKM bases and odd dimensional GKM fibres, or with odd dimensional GKM bases and even dimensional GKM fibres.
\end{rmk}

\subsection{Odd GKM actions on contact and cosymplectic manifolds}
The even GKM theory has many applications to certain torus actions on symplectic manifolds. Here we give some odd-dimensional analogues from contact and cosymplectic manifolds. 

\subsubsection{Odd GKM torus actions on contact manifolds}
Let $(M^{2n+1},\eta)$ be a contact manifold with a $T$-action that preserves the contact form $\eta$. For any $\xi \in \mathfrak{t}$, let $\xi_M$ be the corresponding vector field on $M$, then we have a Hamiltonian function $\mu^\xi\triangleq \eta(\xi_M) \in C^\infty(M)$. 

Unlike the symplectic case as in Example \ref{exm:FormalityOfSymp}, we might not have the equality between the fixed-point set $M^T$ and the critical set $Crit(\mu^\xi)$ for a generic $\xi$ in the contact case. The perfection of $\mu^\xi$, the equivariant formality of $T\curvearrowright M$, and even the non-emptiness of $M^T$ are not guaranteed. For instance, the contact toric actions $T^{n+1}\curvearrowright(M^{2n+1},\eta)$ were shown by Lerman (\cite{Le03} p.\,794, Lem\,2.12) that $M^T=\varnothing$, hence these actions can't be equivariantly formal. 

Suppose an action $T^k\curvearrowright (M^{2n+1},\eta)$ satisfies the odd GKM condition \ref{OddGKMCond} and denote its odd GKM graph by $\mathcal{G}_M$. Let $\lambda \in \mathfrak{t}_\Z$ be an isotropy weight and pick any 3d component $C_\lambda \subset M^{T_\lambda}$. It can be checked that, as a $T_\lambda$-fixed submanifold, $(C_\lambda,\eta|_{C_\lambda})$ is contact and is preserved by the residue $T/T_{\lambda}$-action. Niederkruger (\cite{Ni05} p.\,50, Thm\,IV.16) proved the existence of invariant contact forms on any $S^1$-orientable, closed 3d manifold that has non-empty fixed-points. Hence, the 3d contact $T/T_\lambda$-equivariantly formal manifold $(C_\lambda,\eta|_{C_\lambda})$ can be any of the $S^1$-orientable, equivariantly formal, closed 3d manifold $M^3=\big\{g=0,\epsilon=o,f>0,s=0,(\mu_1,\,\nu_1),\ldots,(\mu_r,\,\nu_r)\big\}$. If we want to reduce the possibilities of $C_\lambda$, we could impose extra conditions.

Suppose an odd GKM action $T^k\curvearrowright (M^{2n+1},\eta)$ can be extended to a larger action $T'=T^{k+1}\curvearrowright (M^{2n+1},\eta)$ such that $T^k$ is identified as $T^k\times\{1\}\subset T'$. Then the equivariantly formal, orientation-preserving action $T/T_{\lambda}\curvearrowright C_\lambda$ can be extended to a larger action $T'/T_{\lambda}\curvearrowright C_\lambda$. By the discussion in Subsubsection \ref{subsubsec:ExtendibleAction} on Neumann's results of $T^2$-actions on orientable 3d manifolds, there are two possibilities for $C_\lambda$: $S6=S^2\times S^1$ and Lens spaces $S7=L(p,q)$, both of which have $T^2$-invariant contact structures by Lerman's classification (\cite{Le03} p.\,796, Thm\,2.18) of toric contact manifolds. Each $C_\lambda=S^2\times S^1$ contributes to the GKM graph $\mathcal{G}_M$ a $\Box$ of weight $\lambda$ joining to two $\circ$-vertices $\gamma_1,\gamma_2$, and contributes to $H^*_T(M)$ the congruence relations on the tuple $(P_{\gamma_1},Q_{\gamma_1}\theta;P_{\gamma_2},Q_{\gamma_2}\theta)$:
\begin{align*}
P_{\gamma_1}\equiv P_{\gamma_2} \qquad Q_{\gamma_1}\equiv Q_{\gamma_2} \mod{\lambda}.
\end{align*} 
Each $C_\lambda=L(p,q)$ contributes a $\Box$ of weight $\lambda$ joining to one $\circ$-vertex $\gamma$ and a congruence relation on the duple $(P_{\gamma},Q_{\gamma}\theta)$: 
\begin{align*}
Q_{\gamma}\equiv 0 \mod{\lambda}.
\end{align*} 
\begin{exm}
	$S^{2n}\times S^1 = \big\{(r_1 e^{i\phi_1},\ldots,r_n e^{i\phi_n},r_{n+1};e^{i\phi_{n+1}})\mid r_1^2+\cdots r^2_{n+1}=1\big\}$ has the contact form $\eta=r_1d\phi_1+\cdots+r_{n+1}d\phi_{n+1}$ invariant under the $T^{n+1}$-action $(e^{i\psi_1},\ldots,e^{i\psi_{n+1}})\Cdot (r_1 e^{i\phi_1},\ldots,r_n e^{i\phi_n},r_{n+1};e^{i\phi_{n+1}}) = (r_1 e^{i(\psi_1+\phi_1)},\ldots,r_n e^{i(\psi_n+\phi_n)},r_{n+1};e^{i(\psi_{n+1}+\phi_{n+1})})$. Since the action $T^n\curvearrowright S^{2n}$ has the even GKM graph consisting of two vertices with $n$ edges of weights $\alpha_1,\ldots,\alpha_n$ as in Example \ref{exm:EvenSphere}, the $T^n\times \{1\}$-action on $S^{2n}\times S^1$ is a product of GKM actions and has the odd GKM graph consisting of two $\circ$-vertices joined via $n$ $\Box$-vertices representing $S^2\times S^1$'s of weights $\alpha_1,\ldots,\alpha_n$. We have the localized expression $H^*_{T^n}(S^{2n}\times S^1)\cong \big\{(P_{N},Q_{N}\theta;P_{S},Q_{S}\theta) \mid P_{N}\equiv P_S, \, Q_{N}\equiv Q_S \mod \prod_{j=1}^{n}\alpha_j\big\}$ which is of course isomorphic to  $H^*_{T^n}(S^{2n})\otimes H^*(S^{1})$.
\end{exm} 

Suppose moreover that the extended action $T'=T^{k+1}\curvearrowright (M^{2n+1},\eta)$ is of \textbf{Reeb type}, i.e. there exists $\xi \in \mathfrak{t}'\setminus \mathfrak{t}$ such that $\eta(\xi_M)=1,\,\iota_{\xi_M}d\eta=0$. It can be checked that, for each 3d component $C_\lambda$, the toric action $T'/T_{\lambda}\curvearrowright (C_\lambda, \eta|_{C_\lambda})$ is also of Reeb type. By (\cite{Lih} p.\,5, Prop\,2.3), such a $C_\lambda$ can only be a Lens space but excluding $S^2\times S^1$. Each $C_\lambda$, being a Lens space, has one single $T/T_\lambda$-fixed circle which is also $T$-fixed. Hence, the odd GKM graph $\mathcal{G}_M$ consists of a single $\circ$-vertex $\gamma$ joined to $n$ $\Box$'s of pairwise independent weights $\lambda_1,\ldots,\lambda_n$, such that $ H^*(\mathcal{G}_M)$ consists of the duples $(P_{\gamma},Q_{\gamma}\theta)$ satisfying the congruence relations
\begin{align*}
Q_{\gamma}\equiv 0 \mod{\lambda_1,\ldots,\lambda_n}.
\end{align*} 
Therefore, we have the $H^*_{T}(pt)$-algebra isomorphisms 
$$H^*_{T}(M^{2n+1})\cong H^*(\mathcal{G}_M)=\big\{(P_{\gamma},Q_{\gamma}\theta) \mid Q_{\gamma}\equiv 0 \mod \prod_{j=1}^{n}\lambda_j\big\} \cong H^*_{T}(pt)\otimes H^*(S^{2n+1}).$$
Restricting to ordinary cohomology, we have $H^*(M^{2n+1})\cong H^*(S^{2n+1})$, i.e. $M^{2n+1}$ has to be a rational homology sphere. 
\begin{exm}
	The sphere $S^{2n+1}=\big\{(r_1 e^{i\phi_1},\ldots,r_n e^{i\phi_n},r_{n+1}e^{i\phi_{n+1}})\mid r_1^2+\cdots r^2_{n+1}=1\big\}$ has the contact form $\eta=r_1d\phi_1+\cdots+r_{n+1}d\phi_{n+1}$ invariant under the $T^{n+1}$-action $(e^{i\psi_1},\ldots,e^{i\psi_{n+1}})\Cdot (r_1 e^{i\phi_1},\ldots,r_n e^{i\phi_n},r_{n+1}e^{i\phi_{n+1}}) = (r_1 e^{i(\psi_1+\phi_1)},\ldots,r_n e^{i(\psi_n+\phi_n)},r_{n+1}e^{i(\psi_{n+1}+\phi_{n+1})})$. It is odd GKM under the action of $T^n\times \{1\}$ concentrating on the first $n$ coordinates. Furthermore, given positive integers $0<q_1,\ldots,q_n<p$ such that each $q_j$ is coprime to $p$, we consider the Lens space $L(p;q_1,\ldots,q_n)$ defined as the quotient of a $\Z/p\Z$-action on $S^{2n+1}$: $e^{2\pi i/p}\Cdot (r_1 e^{i\phi_1},\ldots,r_n e^{i\phi_n},r_{n+1}e^{i\phi_{n+1}})= (r_1 e^{2\pi i q_1/p} \cdot e^{i\phi_1}, \ldots,r_n e^{2\pi i q_n/p} \cdot e^{i\phi_n},r_{n+1} e^{2\pi i /p} \cdot e^{i\phi_{n+1}})$. The lens space $L(p;q_1,\ldots,q_n)$ inherits from $S^{2n+1}$ a contact form and a compatible effective $T^{n+1}$-action whose $T^n\times \{1\}$-subaction is odd GKM with the graph consisting of a single $\circ$-vertex $\gamma$ joined to $n$ $\Box$'s representing the 3-dimensional Lens subspaces $L(p;q_i)$ of the fundamental $T^n$-weights $\alpha_i$.
\end{exm}

\subsubsection{Odd GKM torus actions on cosymplectic manifolds}
Tischler \cite{Ti70} proved that if a compact manifold has a nonvanishing closed one-form, then the manifold is a mapping torus. An odd-dimensional compact manifold $M^{2n+1}$ is a \textbf{cosymplectic manifold} if there are a closed one-form $\eta$ and a closed two-form $\omega$ such that $\eta\wedge \omega^n$ is a volume form. H.-J. Li \cite{Li08} proved that a cosymplectic manifold is a mapping torus of a symplectic manifold.

Suppose a cosymplectic manifold $(M,\eta,\omega)$ is acted by a torus $T$ that preserves the closed forms $\eta,\,\omega$. For every $\xi\in \mathfrak{t}$, let $\xi_M$ be the vector field generated by $\xi$ on $M$. Let $R$ be the \textbf{Reeb vector field} on $M$ defined by $\eta(R)=1$ and $\iota_R \omega =0$. Following Albert \cite{Al89}, Guillemin-Miranda-Pires-Scott \cite{GMPS15,GMPS17} and Bozzoni-Goertsches \cite{BG19}, we call the action $T\curvearrowright (M,\eta,\omega)$ to be \textbf{Hamiltonian} if 
\begin{itemize}
	\item $\eta(\xi_M)=0$,
	\item there exists a moment map $\mu:\mathfrak{t}\rightarrow C^\infty(M)$ such that for every $\xi\in \mathfrak{t}$, the function $\mu^\xi\triangleq\mu(\xi)$ satisfies: 
	\begin{itemize}
		\item $\iota_{\xi_M}\omega = d \mu^\xi$;
		\item $\mu^\xi$ is $T$-invariant;
		\item $\mu^\xi$ is Reeb-invariant, i.e. $\mathcal{L}_R (\mu^\xi)=0$ where $\mathcal{L}_R$ is the Lie derivative along $R$.
	\end{itemize} 
\end{itemize}
Assuming there is a compact leaf of the foliation defined by $\eta$, Guillemin-Miranda-Pires-Scott showed that a $T$-Hamiltonian cosymplectic manifold is a mapping torus of a $T$-Hamiltonian symplectic manifold. When the action is \textbf{toric}, i.e. $2\dim T +1 =\dim M$, they showed that the cosymplectic manifold is a product of $S^1$ with a toric symplectic manifold. Bozzoni-Goertsches gave a new proof of the product structure of a toric cosymplectic manifold without the compactness assumption. 

For a non-toric $T$-Hamiltonian cosymplectic manifold $(M,\eta,\omega)$, suppose $M$ has only isolated $T$-fixed circles, the argument of Bozzoni-Goertsches(\cite{BG19} Prop\,3.4 and Thm\,3.7) still works. Hence the action $T\curvearrowright M$ is equivariantly formal, $b_1(M)=1$ and there is a $T$-equivariant cosymplectomorphism 
\[
M \cong N_\phi = \frac{N\times [0,r]}{N\times\{0\}\sim_\phi N\times\{r\}}
\]
where $N$ is a compact connected manifold equipped with a symplectic form $\omega'$ and a Hamiltonian $T$-action; $\phi: N\rightarrow N$ is a $T$-Hamiltonian symplectomorphism; $r$ is a positive real number. We identify $N$ as the submanifold $N\times\{0\}$ in $M\cong N_\phi$. Since $M$ has only isolated $T$-fixed circles, then $N$ has only isolated $T$-fixed points. 

Suppose further that the $1$-skeleton $M_1$ is 3-dimensional, i.e. the action $T\curvearrowright M$ is odd GKM, then the $1$-skeleton $N_1$ is 2-dimensional, i.e. the action $T\curvearrowright N$ is even GKM. The $T$-Hamiltonian symplectomorphism $\phi$ on $N$ naturally induces an automorphism $\tilde{\phi}$ on the even GKM graph $\mathcal{G}_N$. 

\begin{lem}
	The automorphism $\tilde{\phi}$ on $\mathcal{G}_N$ is the identity map. 
\end{lem}
\begin{proof}
	Because of $M\cong N_\phi$ and the requirement that $\mu$ is Reeb-invariant, the moment map $\mu$ essentially lives on the leaf $N$. The restriction $\mu|_N$, also denoted $\mu_N$, is the moment map for the $T$-Hamiltonian action on $(N,\omega)$. Since $N$ has only isolated fixed points, then by the Atiyah-Guillemin-Sternberg convexity theorem, there is a generic $X\in \mathfrak{t}$ such that $\mu^X_N$ attains its minimum at a single fixed point $p_{min}$ of $N$ and its maximum at a single fixed point $p_{max}$. The $T$-Hamiltonian symplectomorphism $\phi$ preserves $\mu_N$, then $\phi(p_{min/max})=p_{min/max}$ on $N$ and $\tilde{\phi}(p_{min/max})=p_{min/max}$ on the graph $\mathcal{G}_N$. Next, let $p \in N^T$ be a vertex of $\mathcal{G}_N$ joined to $p_{min}$ via a unique edge $\overline{p_{min}p}$ of weight $\alpha$. Since $\phi$ is $T$-equivariant, then $\tilde{\phi}$ preserves the weights of the edges on $\mathcal{G}_N$. We already have $\tilde{\phi}(p_{min})=p_{min}$, then $\tilde{\phi}(\overline{p_{min}p})=\overline{p_{min}p}$, hence $\tilde{\phi}(p)=p$. Let $\overline{p_{min}p_1\cdots p_lp_{max}}$ be a path in $\mathcal{G}_N$, then we will have $\tilde{\phi}(\overline{p_{min}p_1\cdots p_lp_{max}})=\overline{p_{min}p_1\cdots p_lp_{max}}$. Note $N$ is connected, then $\mathcal{G}_N$ is connected, hence every edge appears in some path between $p_{min}$ and $p_{max}$. Therefore, $\tilde{\phi}$ fixes every edge and every vertex of $\mathcal{G}_N$.
\end{proof}
\noindent The lemma is equivalent to saying that $\phi$ restricted on the 1-skeleton $N_1$ is isotopic to the identity map $id_{N_1}$. Then we have $M_1\cong {(N_1)}_{\phi}\cong {(N_1)}_{id}=N_1\times S^1$. Therefore, the odd GKM graphs of $M$ and $N\times S^1$ are the same, and so are their $T$-equivariant cohomology rings:
\[
\mathcal{G}_M\cong \mathcal{G}_{N\times S^1}\qquad \qquad H^*_T(M)\cong H^*_T(N\times S^1) \cong H^*_T(N)\otimes H^*(S^1).
\] 

\begin{rmk}
	Without assuming the cosymplectic structure, we can consider the mapping tori of even GKM manifolds directly. Let $\phi:N\rightarrow N$ be an equivariant diffeomorphism on an even GKM, possibly non-orientable, $T$-manifold $N$. The manifold $N$ does not have to be orientable, neither does the mapping torus $N_\phi$ nor its 3d $1$-skeleton $(N_\phi)_1=(N_1)_\phi$. By Subsubsection \ref{subsubsec:MapTori} on mapping tori of 2d $S^1$-manifolds, a component of $(N_1)_\phi$ could be of the orientable type $S^2\times S^1$, or the non-orientable types $S^2\times_{\Z_2}S^1,\, \R P^2 \times S^1$.  
\end{rmk}

\subsection{Odd-dimensional Grassmannians}
	The Grassmannians $G_{2k+1}(\R^{2n+2})$ and $\tilde{G}_{2k+1}(\R^{2n+2})$ are of the odd dimension $(2k+1)(2n-2k+1)$. They are equipped with certain canonical odd GKM $T^n$-actions that commute with the $\Z_2$-cover $\Z_2 \rightarrow \tilde{G}_k(\R^n) \rightarrow G_k(\R^n)$. It turns out that the odd GKM graphs of $G_{2k+1}(\R^{2n+2})$ and $\tilde{G}_{2k+1}(\R^{2n+2})$ are the same, and are closely related with the even GKM graph of certain canonical $T^n$-action on $G_{2k}(\R^{2n})$.
	\begin{exm}
		Let $T^2$ act on $\R^6$ such that the two $S^1$-factors of $T^2$ respectively rotate the first two $\R^2$-summands of $\R^6=\R_1^2\oplus\R_2^2\oplus \R_0^2$. This real $T^2$-representation induces a canonical $T^2$-action on $G_3(\R^6)$ whose fixed-points are real $3$-dimensional $T^2$-subrepresentations of the representation $T^2\curvearrowright \R^6$. The fixed points form two connected components $\gamma_1=\{\R_1^2\oplus L \in G_3(\R^6) \mid  L\in \P(\R_0^2)\},\,\gamma_2=\{\R_2^2\oplus L \in G_3(\R^6) \mid  L\in \P(\R_0^2)\}$, both of which are parametrized by a $\P(\R_0^2)=\R \P^1\cong S^1$. Similarly, there is an induced $T^2$-action on $\tilde{G}_3(\R^6)$ whose fixed points also form two connected components $\tilde{\gamma}_1=\{\R_1^2\oplus L \in \tilde{G}_3(\R^6) \mid  L\in \tilde{G}_1(\R_0^2)\},\,\tilde{\gamma}_2=\{\R_2^2\oplus L \in \tilde{G}_3(\R^6) \mid  L\in \tilde{G}_1(\R_0^2)\}$, both of which are parametrized by a $\tilde{G}_1(\R_0^2)=S^1$. To understand the $1$-skeleta, one can consider the $3$-dimensional real $S^1$-subrepresentations of $\R^6$. The construction is straightforward though tedious, see \cite{He}. Let $\alpha_1,\alpha_2$ be the standard integral basis of the dual Lie algebra of $T^2$. It turns out that the odd GKM graphs of $G_3(\R^6)$ and $\tilde{G}_3(\R^6)$ (Figure \ref{fig:oddGrass}) are the same. For the graph of $\tilde{G}_{3}(\R^{6})$, the four corner $\Box$'s represent $S^3$ and the two middle $\Box$'s represent $S^2\times S^1$ (\cite{He} p.22 Prop 6.6 (4)). For the graph of ${G}_{3}(\R^{6})$, these $\Box$'s represent $\R P^3$ and $S^2\times S^1$ respectively (\cite{He} p.11 (3)). 
		\begin{figure}[H]
			\centering
			\begin{subfigure}[b]{0.45\textwidth}
				\centering
				\begin{tikzpicture}[scale=1]
				\path (-1,0) node[draw,circle,label=:{$\tilde{\gamma}_1$}] (v1){}
				(1,0) node[draw,circle,label=:{$\tilde{\gamma}_2$}] (v2){}
				(0,1) node[draw,rectangle,scale=1.2,label={$\alpha_2+\alpha_1$}] (v3){}
				(0,-1) node[draw,rectangle,scale=1.2,label=below:{$\alpha_2-\alpha_1$}] (v4){}
				(-2,1) node[draw,rectangle,scale=1.2,label={$\alpha_1$}] (v5){}
				(-2,-1) node[draw,rectangle,scale=1.2,label=below:{$\alpha_2$}] (v6){}
				(2,1) node[draw,rectangle,scale=1.2,label={$\alpha_1$}] (v7){}
				(2,-1) node[draw,rectangle,scale=1.2,label=below:{$\alpha_2$}] (v8){};
				
				\draw (v1) -- node {} (v5)
				(v1) -- node {} (v6)
				(v1) -- node {} (v3)
				(v1) -- node {} (v4)
				(v2) -- node {} (v3)
				(v2) -- node {} (v4)
				(v2) -- node {} (v7)
				(v2) -- node {} (v8);
				\end{tikzpicture}
				\caption{Odd GKM Graph of $\tilde{G}_{3}(\R^{6})$}
			\end{subfigure}
			\begin{subfigure}[b]{0.45\textwidth}
				\centering
				\begin{tikzpicture}[scale=1]
				\path (-1,0) node[draw,circle,label=:{$\gamma_1$}] (v1){}
				(1,0) node[draw,circle,label=:{$\gamma_2$}] (v2){}
				(0,1) node[draw,rectangle,scale=1.2,label={$\alpha_2+\alpha_1$}] (v3){}
				(0,-1) node[draw,rectangle,scale=1.2,label=below:{$\alpha_2-\alpha_1$}] (v4){}
				(-2,1) node[draw,rectangle,scale=1.2,label={$\alpha_1$}] (v5){}
				(-2,-1) node[draw,rectangle,scale=1.2,label=below:{$\alpha_2$}] (v6){}
				(2,1) node[draw,rectangle,scale=1.2,label={$\alpha_1$}] (v7){}
				(2,-1) node[draw,rectangle,scale=1.2,label=below:{$\alpha_2$}] (v8){};
				
				\draw (v1) -- node {} (v5)
				(v1) -- node {} (v6)
				(v1) -- node {} (v3)
				(v1) -- node {} (v4)
				(v2) -- node {} (v3)
				(v2) -- node {} (v4)
				(v2) -- node {} (v7)
				(v2) -- node {} (v8);
				\end{tikzpicture}
				\caption{Odd GKM Graph of $G_{3}(\R^{6})$}
			\end{subfigure}
			\caption{Odd GKM Graphs of odd-dim real and oriented Grassmannians.}\label{fig:oddGrass}
		\end{figure}
		\noindent By Theorem \ref{thm:OddGKM}, every equivariant cohomology class of $\tilde{G}_{3}(\R^{6}),\,G_{3}(\R^{6})$ is a tuple $(f_1,g_1\theta;f_2,g_2\theta)$ where $f_i,g_i \in \Q[\alpha_1,\alpha_2]$ satisfy the congruence relations:
		\begin{align*}
		\begin{cases}
		g_1 \equiv 0  \quad  &g_2 \equiv 0 \quad \mod \alpha_1\\
		g_1 \equiv 0  \quad  &g_2 \equiv 0 \quad \mod \alpha_2\\
		f_1 \equiv f_2  \quad &g_1  \equiv g_2 \quad \mod \alpha_2+\alpha_1\\
		f_1 \equiv f_2  \quad &g_1  \equiv g_2 \quad \mod \alpha_2-\alpha_1.
		\end{cases}	
		\end{align*} 
		The first two congruence relations mean that we can write $g_1=h_1\alpha_1\alpha_2,\,g_2=h_2\alpha_1\alpha_2$ for $h_1,h_2 \in \Q[\alpha_1,\alpha_2]$. Since $\alpha_1\alpha_2$ is coprime with $\alpha_2\pm\alpha_1$, after plugging the $h$-expressions of $g_1,g_2$ into the last two congruence relations, we see that $h_1,h_2$ share the same congruence relations with $f_1,f_2$, which can be shown to be exactly the congruence relations of certain canonical even GKM $T^2$-action on $G_2(\R^4)$. Therefore the correspondence $(f_1,g_1\theta;f_2,g_2\theta)\mapsto (f_1,f_2;h_1\alpha_1\alpha_2\theta,h_2\alpha_1\alpha_2\theta)$ gives a ring isomorphism $H^*_{T^2}(\tilde{G}_3(\R^6))\cong H^*_{T^2}(G_3(\R^6))\cong H^*_{T^2}(G_2(\R^4))[r]/r^2$ where $r=\alpha_1\alpha_2\theta$ is of degree $5$.
	\end{exm}
	
	For the general odd-dimensional real and oriented Grassmannians, the details of localizing the equivariant cohomology rings can be found in \cite{He}.

\subsection{Torus subaction of certain cohomogeneity-one action}
	An action $G\curvearrowright M$ is of \textbf{cohomogeneity one} if the quotient space $M/G$ is one-dimensional. In the following, we suppose $M/G$ is an interval $[-1,1]$. By the differentiable slice theorem, Mostert \cite{Mo57} proved that, over the open interval $(-1,1)$ we have an open cylinder $G/H \times (-1,1)$, and over the two endpoints $\{\pm 1\}$ we have $G/K_\pm$ such that $K_\pm \supseteq H$ and $K_\pm/H$ are spheres. Conversely, any sequence of compact Lie groups $G\supseteq K_\pm \supseteq H$ such that $K_\pm/H$ are spheres $S^{n_\pm}$ produce a cohomogeneity-one $G$-manifold with an interval orbit space by forming $M= G/K_- \cup_{\pi_-} (G/H \times (-1,1)) \cup_{\pi_+} G/K_+$ where the gluing takes place at the two ends via the $K_\pm/H\cong S^{n_\pm}$-bundle projections $\pi_\pm: G/H\times \{\pm 1\}\rightarrow G/K_\pm$.
	
	Let us assume $\mathrm{rank}\,G = \mathrm{rank}\,H$. Such a cohomogeneity-one manifold $M=(G,K_+,K_-,H)$ is odd dimensional and is equivariantly formal with respect to the subaction of a maximal torus $T\subseteq H$ by the results of Goertsches and Mare (\cite{GM14b} p.\,37, Cor\,1.3). Note that $T\curvearrowright G/H$ is even GKM. Moreover, for the open dense part $G/H \times (-1,1)\subset M$, the $T$-fixed-point set $(G/H \times (-1,1))^T=(G/H)^T \times (-1,1)$ consists of finite number of intervals whose isotropy weights are inherited from the GKM data of weights of $(G/H)^T$ as in \cite[p.\,28 Thm 2.4]{GHZ06} and Prop.\,\ref{prop:GHZ}. These intervals in $(G/H)^T \times (-1,1)$ are glued at the singular part ${M\setminus (G/H \times (-1,1)) = G/K_- \cup G/K_+}$, whose $T$-actions are also even GKM since $\mathrm{rank}\,G = \mathrm{rank}\,K_\pm$. Therefore, the $T$-action on $M$ is odd GKM. This observation was kindly pointed out by Goertsches and Mare to the author.
	
	By (\cite{GHZ06} p.\,28, (2.9)), in the 1-skeleton of $G/H$, each 2d component with isotropy weight $\alpha$  is acted transitively by a rank-1 semisimple subgroup $G_\alpha \subseteq G$ which is isomorphic to either $SU(2)$ or $SO(3)$. Hence each 3d component in the 1-skeleton of $M$ is acted by an $SU(2)$ or $SO(3)$ nontrivially and nontransitively. By the discussions in Subsubsection \ref{subsubsec:ExtendibleAction}, such an action of $SU(2)$ or $SO(3)$ can be reduced to an effective cohomogeneity-one action of $SO(3)$, and the 3d component is of one of the six equivariant diffeomorphism types: $S8=\R P^2 \times S^1,\,S9'=S^2\times_{\Z_2} S^1,\,S9=S^2\times S^1$ and $S10=\R P^3 \# \R P^3,\,S11= S^3,\, S12=\R P^3$.  
	\begin{exm}
		Consider the $7$-dimensional cohomogeneity-one manifold $(G=U(3),K_-=K_+=U(2)\times U(1),H=U(1)^3)$, which is actually the manifold $N_G^7$ in Hoelscher's classification (\cite{Ho10a} p.131, \cite{Ho10b} p.170). 
		By the above discussion, we need to understand the $1$-skeleta and the GKM graphs of $G/H=Fl(3)$ and $G/K_\pm = \C P^2$ under the canonical $U(1)^3=T^3$-actions of left multiplications. Though the $T^3$-actions are not effective, we can still apply the GKM theories. Let $\alpha_1,\alpha_2,\alpha_3$ be the standard integral basis of the dual Lie algebra of $T^3$. By the results of \cite{GHZ06} and with more details in Sabatini's PhD thesis \cite{Sa09} p.48-49, the GKM graph of $G/H=Fl(3)$ has the symmetric group $S_3$ as vertex set. If two permutations differ by a transposition $(i,j)$, then they are joined via an edge of weight $\alpha_j-\alpha_i$. The GKM graph of $G/K_\pm = \C P^2$ has $\{1,2,3\}$ as vertex set. The vertices $i$ and $j$ are joined via an edge of weight $\alpha_j-\alpha_i$. The fibration projection from the GKM graph of $Fl(3)$ to that of $\C P^2$ is (our notations are slightly different from \cite{Sa09}): the edge joining $(i,j,k)$ to $(j,i,k)$ projects to the point $k$ (see the three thick edges in Figure \ref{fig:P2fl} (A)); the other edges project accordingly to edges.
		\begin{figure}[H]
			\centering 
			\begin{subfigure}[b]{0.45\textwidth}
				\centering
				\begin{tikzpicture}[scale=1.2]
				\path 
				(65:1.5cm) node[draw,fill,circle,inner sep=0pt,minimum size=4pt,label=above:{$(213)$}] (v1) {}
				(115:1.5cm) node[draw,fill,circle,inner sep=0pt,minimum size=4pt,label=above:{$(123)$}] (v2) {}
				(-175:1.5cm) node[draw,fill,circle,inner sep=0pt,minimum size=4pt,label=left:{$(132)$}] (v3) {}
				(-125:1.5cm) node[draw,fill,circle,inner sep=0pt,minimum size=4pt,label=left:{$(312)$}] (v4) {}
				(-55:1.5cm) node[draw,fill,circle,inner sep=0pt,minimum size=4pt,label=right:{$(321)$}] (v5) {}
				(-5:1.5cm) node[draw,fill,circle,inner sep=0pt,minimum size=4pt,label=right:{$(231)$}] (v6) {};
				
				\path[-] 
				(v1) edge[blue,ultra thick] node[below] {$\alpha_2-\alpha_1$} (v2)
				(v3) edge[blue] (v6)
				(v4) edge[blue] node[above] {$\alpha_2-\alpha_1$} (v5)
				(v3) edge[green,ultra thick] node[left] {$\alpha_3-\alpha_1$} (v4)
				(v2) edge[green] (v5)
				(v1) edge[green] node[right] {$\alpha_3-\alpha_1$} (v6)
				(v5) edge[red,ultra thick] node[right] {$\alpha_3-\alpha_2$} (v6)
				(v1) edge[red] (v4)
				(v2) edge[red] node[left] {$\alpha_3-\alpha_2$} (v3);
				\end{tikzpicture}
				\caption{GKM Graph of $Fl(3)$}
			\end{subfigure}
			\begin{subfigure}[b]{0.45\textwidth}
				\centering
				\begin{tikzpicture}[scale=1.2]
				\path 
				(90:1.5cm) node[draw,fill,circle,inner sep=0pt,minimum size=4pt,label=above:{$3$}] (w3) {}
				(-150:1.5cm) node[draw,fill,circle,inner sep=0pt,minimum size=4pt,label=left:{$2$}] (w2) {}
				(-30:1.5cm) node[draw,fill,circle,inner sep=0pt,minimum size=4pt,label=right:{$1$}] (w1) {};
				
				\path[-] 
				(w3) edge[red] node[left] {$\alpha_3-\alpha_2$} (w2)
				(w3) edge[green] node[right] {$\alpha_3-\alpha_1$} (w1)
				(w1) edge[blue] node[above] {$\alpha_2-\alpha_1$} (w2);
				\end{tikzpicture}
				\caption{GKM Graph of $\C P^2$}
			\end{subfigure}
			\caption{GKM Graphs of $G/H=Fl(3)$ and $G/K_\pm = \C P^2$.}\label{fig:P2fl}
		\end{figure}
		\noindent The GKM bundles $K_\pm/H\rightarrow G/H\rightarrow G/K_\pm$ are both of the form $\C P^1 \rightarrow Fl(3) \overset{\pi}{\rightarrow} \C P^2$. The $1$-skeleton of $N_G^7$ is understood as follows. First, the manifold $N_G^7$ is formed by taking the cylinder $Fl(3)\times[-1,1]$ and collapsing the two ends $Fl(3)\times \{\pm 1\}$ to two $\C \P^2$'s via the map $\pi$ above. Second, let $Fl(3)_1$ be the $1$-skeleton of $Fl(3)$. The $1$-skeleton of ${N_G^7}$ is formed from the cylinder $Fl(3)_1\times [-1,1]$ by projecting the two ends $Fl(3)_1\times \{\pm 1\}$ via $\pi$. Third, take an $S^2$ in the $1$-skeleton of $Fl(3)$, then it lifts to an $S^2\times [-1,1]$ in the $1$-skeleton of $Fl(3)\times [-1,1]$. If this $S^2$ projects to an $S^2$ in the $1$-skeleton of $\C P^2$, then the two ends $S^2\times \{\pm 1\} \subset S^2\times [-1,1]$ survive in $M$ after projections. Hence the corresponding $S^2\times [-1,1]$ is glued with another $S^2\times [-1,1]$ at the two ends to form an $S^2\times S^1$ in $M$. If this $S^2$ projects to a fixed point of $\C P^2$, then the corresponding $S^2\times [-1,1]$ gets collapsed at the two ends $S^2\times \{\pm 1\}$ to two points and results in an $S^3$ in $M$. The $1$-skeleton graph of $N_G^7$ is drawn in the following Figure\,\ref{fig:NG7}. The $\circ$ labelled with $k$ is resulted from the pair of vertices $(ijk),(jik)$ of $\mathcal{G}_{Fl(3)}$. The inner three $\Box$'s represent $S^2\times S^1$, and the outer three $\Box$'s represent $S^3$. 
		\begin{figure}[H]{}
			\centering
			\begin{tikzpicture}
			\path 
			(90:1.5cm) node[draw,circle,label=right:{$3$}] (w3) {}
			(-150:1.5cm) node[draw,circle,label=left:{$2$}] (w2) {}
			(-30:1.5cm) node[draw,circle,label=right:{$1$}] (w1) {}
			(30:0.75cm) node[draw,rectangle,scale=1.2,green,label=right:{\color{green}$\alpha_3-\alpha_1$}] (v31){}
			(150:0.75cm) node[draw,rectangle,scale=1.2,red,label=left:{\color{red}$\alpha_3-\alpha_2$}] (v32){}
			(-90:0.75cm) node[draw,rectangle,scale=1.2,blue,label=below:{\color{blue}$\alpha_2-\alpha_1$}] (v21){}
			(90:2.25cm) node[draw,rectangle,scale=1.2,blue,label={\color{blue}$\alpha_2-\alpha_1$}] (v3){}
			(-150:2.25cm) node[draw,rectangle,scale=1.2,green,label=left:{\color{green}$\alpha_3-\alpha_1$}] (v2){}
			(-30:2.25cm) node[draw,rectangle,scale=1.2,red,label=right:{\color{red}$\alpha_3-\alpha_2$}] (v1){};
			
			\path[-] 
			(w3) edge[green] (v31)
			(w3) edge[red] (v32)
			(w3) edge[blue] (v3)
			(w1) edge[green] (v31)
			(w1) edge[blue] (v21)
			(w1) edge[red] (v1)
			(w2) edge[blue] (v21)
			(w2) edge[red] (v32)
			(w2) edge[green] (v2);
			\end{tikzpicture}
			\caption{$1$-skeleton Graph of $N_G^7$.}\label{fig:NG7}
		\end{figure}
		\noindent By Theorem \ref{thm:OddGKM}, every $T^3$-equivariant cohomology class of $N_G^7$ is a tuple $(f_1,g_1\theta;f_2,g_2\theta;f_3,g_3\theta)$ where $f_i,g_i \in \Q[\alpha_1,\alpha_2,\alpha_3]$ and satisfy the congruence relations:
		\begin{align*}
		\begin{cases}
		g_1 \equiv 0  \qquad\qquad\qquad &\mod \alpha_3-\alpha_2\\
		g_2 \equiv 0  \qquad\qquad\qquad &\mod \alpha_3-\alpha_1\\
		g_3 \equiv 0  \qquad\qquad\qquad &\mod \alpha_2-\alpha_1\\
		f_1 \equiv f_2  \qquad g_1  \equiv g_2 \quad &\mod \alpha_2-\alpha_1\\
		f_2 \equiv f_3  \qquad g_2  \equiv g_3 \quad &\mod \alpha_3-\alpha_2\\
		f_1 \equiv f_3  \qquad g_1  \equiv g_3 \quad &\mod \alpha_3-\alpha_1.
		\end{cases}	
		\end{align*}  
		The first three congruence relations mean that we can write $g_1=h_1(\alpha_3-\alpha_2),\,g_2=h_2(\alpha_3-\alpha_1),\,g_3=h_3(\alpha_2-\alpha_1)$ for $h_1,h_2,h_3 \in \Q[\alpha_1,\alpha_2,\alpha_3]$. Plugging the $h$-expressions of $g_1,g_2$ into the forth congruence relation, and noting that $\alpha_3-\alpha_2,\alpha_3-\alpha_1,\alpha_2-\alpha_1$ are pairwise coprime and $\alpha_3-\alpha_2\equiv \alpha_3-\alpha_1 \mod \alpha_2-\alpha_1$, we then get $h_1\equiv h_2 \mod \alpha_2-\alpha_1$. Likewise, we can plug $g_1,g_2,g_3$ into the fifth and the sixth congruence relations and will see that $h_1,h_2,h_3$ satisfy the same congruence relations as $f_1,f_2,f_3$, which is based on the GKM graph of $\C P^2$. The correspondence $(f_1,g_1\theta;f_2,g_2\theta;f_3,g_3\theta)\mapsto (f_1,f_2,f_3;h_1 r,h_2 r, h_3 r)$, where $r$ is of degree $3$, gives ring isomorphisms $H^*_{T^3}(N_G^7)\cong H^*_{T^3}(\C P^2)[r]/r^2 \cong H^*_{T^3}(\C P^2) \otimes H^*(S^3)$ matching Hoelscher's description (\cite{Ho10a} p.\,131) of the group structure of $H^*(N_G^7,\Z)$.
	\end{exm}
	
	\begin{rmk}
		One of the anonymous referees suggested viewing the above discussed cohomogeneity-one manifold $(G=U(3),K_-=K_+=U(2)\times U(1),H=U(1)^3)$ as $U(3)\times_{U(2)\times U(1)} S^3$. This very useful observation actually applies to any compact cohomogeneity-one manifold of the type $(G,K_-=K_+,H)$ as follows (without assuming $\mathrm{rank}\,G=\mathrm{rank}\,H$). Note that $K_\pm/H\cong S^{n_\pm}$. Since $K_-=K_+$, we denote $K_\pm$ by $K_0$ and $n_\pm$ by $n_0$. For the construction $M=G/K_- \cup_{\pi_-} (G/H \times (-1,1)) \cup_{\pi_+} G/K_+$, we have $G/H \times (-1,1)\cong G \times_{K_0} K_0/H \times (-1,1)\cong G \times_{K_0} S^{n_0} \times (-1,1)$, while gluing $G/H \times (-1,1)$ at its two ends to $G/K_- \cup G/K_+$ means we shall collapse $\big(G \times_{K_0} S^{n_0} \times \{-1\}\big) \cup \big(G \times_{K_0} S^{n_0} \times\{1\}\big)$ to $G/K_0 \cup G/K_0$. Therefore, we have 
		$$M\cong G \times_{K_0} S^{n_0+1}$$
		where $S^{n_0+1}$ is formed from $S^{n_0}\times [-1,1]$ by collapsing the two ends to two points, and the $K_0$-action on $S^{n_0+1}$ is induced from $S^{n_0}\times [-1,1]\cong K_0/H \times [-1,1]$. As a result, we immediately see $H^*_G(M)\cong H^*_G(G \times_{K_0} S^{n_0+1})\cong H^*_{K_0}(S^{n_0+1})$. However, the method of odd GKM theory does not assume $K_-=K_+$.
	\end{rmk}

	For the general case of a cohomogeneity-one $G$-manifold $M$ with $\mathrm{rank}\,G = \mathrm{rank}\,K_\pm = \mathrm{rank}\,H$, we can also carefully analyse the $1$-skeleta and apply the odd GKM theory to localize the equivariant cohomology rings. For the more general case of cohomogeneity-one $G$-manifolds regardless of the ranks and equivariant formality, in a joint work with Carlson, Goertsches and Mare \cite{CGHM19}, we described the equivariant cohomology rings using the representation theory of finite dihedral groups. The joint paper assumes certain orientability \cite[p.\,212 Rmk 4.5, p.\,216 Rmk 5.4]{CGHM19}, while the method of odd GKM theory does not.

\vskip 20pt
\bibliographystyle{amsalpha}

\end{document}